\documentclass[11pt,a4paper,leqno]{article}

\usepackage[T1]{fontenc}
\usepackage[utf8]{inputenc}
\usepackage[french,german,english]{babel}

\title{Path properties of the solution to the stochastic heat equation with Lévy noise}
\author{Carsten Chong\thanks{Institut de mathématiques, École Polytechnique Fédérale de Lausanne, Station 8, CH-1015 Lausanne, e-mails: carsten.chong@epfl.ch, robert.dalang@epfl.ch
\newline		
\indent \, This research is partially supported by the Swiss National Foundation for Scientific Research.}~, Robert C.~Dalang$^\ast$ and Thomas Humeau$^\ast$}

\usepackage{amsmath}
\usepackage{amsfonts}
\usepackage{amssymb}
\usepackage{multicol}
\usepackage{color}
\usepackage{dsfont}
\usepackage{graphicx}
\usepackage{listings}
\usepackage{hyperref}
\usepackage{comment}
\usepackage{amsthm}
\usepackage{soul}
\usepackage{stmaryrd}
\usepackage{fancyhdr}
\usepackage{comment}
\usepackage{enumitem}
\setenumerate{leftmargin=*}
\allowdisplaybreaks

\newcommand{\hs}{H_{r,{loc}}(\R^d)}
\newcommand{\R}{\mathbb{R}}

\newcommand{\lp}{\left(}
\newcommand{\rp}{\right)}
\newcommand{\rc}{\right ]}
\newcommand{\lc}{\left [}
\newcommand{\N}{\mathbb{N}}
\newcommand{\C}{\mathbb{C}}
\newcommand{\Z}{\mathbb{Z}}
\newcommand{\Q}{\mathbb{Q}}
\newcommand{\bbp}{\mathbb{P}}

\newcommand{\Scd}{\mathcal S(\R^d)}
\newcommand{\I}{\leqslant}
\newcommand{\eps}{\varepsilon}
\newcommand{\s}{\geqslant}

\newcommand{\E}{\mathbb{E}}

\newcommand{\PR}{\mathbb{P}}

\newcommand{\temd}{\mathcal S'(\R^d)}
\newcommand{\m}{\mathcal}

\newcommand{\dd}{\mathrm{d}}

\newcommand{\F}{\mathcal F}
\newcommand{\calp}{\mathcal P}
\newcommand{\cals}{\mathcal S}
\numberwithin{equation}{section}
\newtheorem{theo}{Theorem}[section]
\newtheorem{prop}[theo]{Proposition}
\newtheorem{rem}[theo]{Remark}
\newtheorem{df}[theo]{Definition}
\newtheorem{lem}[theo]{Lemma}

\newcommand{\scal}[2]{\left \langle #1 , #2 \right \rangle }

\DeclareUnicodeCharacter{00A0}{ }

\vspace{1cm}

\begin{document}
	\date{}
	\maketitle

	\begin{abstract}
		We consider sample path properties of the solution to the stochastic heat equation, in $\R^d$ or bounded domains of $\R^d$, driven by a Lévy space--time white noise. When viewed as a stochastic process in time with values in an infinite-dimensional space, the solution is shown to have a càdlàg modification in fractional Sobolev spaces of index less than $-\frac d 2$. Concerning the partial regularity of the solution in time or space when the other variable is fixed, we determine critical values for the Blumenthal--Getoor index of the Lévy noise such that noises with a smaller index entail continuous sample paths, while Lévy noises with a larger index entail sample paths that are unbounded on any non-empty open subset. Our results apply to additive as well as multiplicative Lévy noises, and to light- as well as heavy-tailed jumps.
	\end{abstract}
	
	\vfill
	

	\vspace{1cm}
	
	
	\vspace{0.5cm}
	
	\newpage

\section{Introduction}

Let $T>0$ and consider, on a stochastic basis $(\Omega,\mathcal F,(\mathcal F_t)_{t\in[0,T]},\bbp)$ satisfying the usual conditions, 
the stochastic heat equation driven by a Lévy space--time white noise on $[0,T]\times D$ with Dirichlet boundary conditions:
\begin{equation}
	\label{SHE}
	\left\{\begin{array}{ll} \frac{\partial u}{\partial t}(t,x) = \Delta u(t,x)+\sigma(u(t,x)) \dot L(t, x)\, , & (t,x) \in (0,T)\times D \, ,\\ u(t,x)=0\, , & \text{for all } (t,x)\in [0,T]\times\partial D \, ,  \\
		u(0,x)=u_0(x)\, , & \text{for all } x\in D\, , \end{array}\right.
\end{equation}
where $D$ is the whole space $\R^d$ or a bounded domain in $\R^d$, $\sigma\colon \R\to \R$ is a Lipschitz function, $u_0\colon \bar D\to\R$ is a bounded continuous initial condition vanishing on $\partial D$, and $\dot L$ is a Lévy space--time white noise. If $D=\R^d$, the boundary conditions on $u$ and $u_0$ are considered void. 

A predictable random field 
$u=\lp u(t,x)\colon (t,x) \in [0,T]\times D \rp$ is called a \emph{mild solution} to \eqref{SHE} if for all $(t,x) \in  [0,T]\times D$,
\begin{equation}
	\label{mildsolution}
	u(t,x)=V(t,x)+\int_0^t \int_D G_D(t-s;x,y) \sigma(u(s,y)) \,L(\dd s , \dd y)\, 
\end{equation}
almost surely, where 
\begin{equation}\label{initial-condition}
	V(t,x)=\int_D G_D(t;x,y)u_0(y)\,\dd y\,,\quad (t,x)\in[0,T]\times D\,,
\end{equation}
is the solution to the homogeneous version of \eqref{SHE}. 

In \eqref{mildsolution} and \eqref{initial-condition}, $G_D$ denotes the \emph{Green's function} of the heat operator on $D$, which for $D=\R^d$ equals the Gaussian density
\begin{equation} \label{heat-kernel} g(t,x) = (4\pi t)^{-\frac d 2}e^{-\frac{|x|^2}{4t}}\mathds 1_{t\s 0} \end{equation}
(when $t=0$, we interpret $g(0,x)$ as the Dirac delta function $\delta_0(x)$),
while on a bounded domain $D$ with smooth boundary it has the spectral representation
\begin{equation}\label{GD-spectral}
	G_D(t;x,y)=\sum_{j\s 1} \Phi_j(x) \Phi_j (y) e^{-\lambda_j t}\mathds 1_{t\s 0}\, ,  \qquad \text{for all} \ x,y \in D\, ,
\end{equation}
where $(\lambda_j)_{j\s 1}$ are the eigenvalues of $-\Delta$ with vanishing Dirichlet boundary conditions, and $(\Phi_j)_{j\s 1}$ are the corresponding eigenfunctions forming a complete orthonormal basis of $L^2(D)$.

In the special case where $\dot L$ is a \emph{Gaussian} noise, the existence, uniqueness and regularity of solutions to Equation~\eqref{SHE} have been extensively studied in the literature, see e.g.\ \cite{bally,chen_dalang,davar, spdewalsh} for the case of space--time white noise, \cite{dalang99, sanzsole02, sanzsole03} for noises that are white in time but colored in space, and \cite{nualart_huang} for noises that may exhibit temporal covariances as well. In all cases, the mild solution to \eqref{mildsolution} is jointly locally Hölder continuous in space and time, with exponents that depend on the covariance structure of the noise.

By contrast, suppose that $\dot L$ is a \emph{Lévy space--time white noise without Gaussian part}, that is, 
\begin{equation}\label{noise}
	\begin{aligned}
		L( \dd t, \dd x)&= b\,\dd t\,\dd x+\int_{|z|\I 1} z\, \tilde J(\dd t, \dd x, \dd z)+\int_{|z|> 1} z  \,J(\dd t, \dd x, \dd z)\\
		&=:L^B(\dd t,\dd x)+L^M(\dd t, \dd x) +L^P(\dd t , \dd x)\, ,
	\end{aligned}
\end{equation}
where  {$b\in\R$}, $J$ is an $(\mathcal F_t)_{t\in[0,T]}$-Poisson random measure on $[0,T]\times D \times \R$ with intensity $\dd t \, \dd x \, \nu(\dd z)$, and $\tilde J$ is the compensated version of $J$. Here $\nu$ is a
\emph{Lévy measure}, that is, $\nu(\{0\}) =0$ and $\int_\R \lp z^2\wedge 1 \rp \,\nu(\dd z)<+\infty$, and we assume that $\nu$ is not identically zero. The existence and uniqueness of solutions for equations like \eqref{SHE} with Lévy noise have been investigated in \cite{ Balan14, Balan15, Chong2016, Chongheavytailed, peszat, loubert}.

Already in the linear case with $\sigma(x)\equiv1$, due to the singularity of the Green's kernel on the diagonal $x=y$ near $t=0$, each jump of the noise creates a Dirac mass for the solution. Even worse, if $\nu(\R)=\infty$, these space--time jump points form a dense subset of $[0,T]\times D$. Hence one cannot expect the solution to have any continuity properties jointly in space and time. 

In this article, we thus take two different viewpoints and consider
\begin{enumerate}
	\item the path properties of $t\mapsto u(t,\cdot)$ as a process with values in an infinite-dimensional space;
	\item the path properties of the partial maps $t\mapsto u(t,x)$ for fixed $x\in D$, and of $x\mapsto u(t,x)$ for fixed $t\in[0,T]$.
\end{enumerate}  

For each $t\s 0$, $u(t,\cdot)$ may take values in $L^p(D)$ for some $p>0$ almost surely, but since
 each atom of the Lévy noise introduces a Dirac delta into the solution, the process $t\mapsto u(t,\cdot)$ cannot have a càdlàg version in such a space (see also \cite{Brzezniak10b} or \cite[Proposition~9.25]{peszat}). Instead, one should consider spaces of distributions containing delta functions, such as negative fractional Sobolev spaces $H_r(D)$ for $r<-\frac d 2$ (see Sections~\ref{sobolevspace}, \ref{frac_sobolev_space} and \ref{fracsobolevdef}). If $\sigma=1$ and the noise has a finite second moment, the existence of a càdlàg modification in such spaces follows from a result of \cite{Kotelenez82} on maximal inequalities for stochastic convolutions in an infinite-dimensional setting, see also \cite[Chapter~9.4.2]{peszat}. This type of result has also been obtained in the case of additive (possibly colored) Lévy noise in \cite{Brzezniak09,Brzezniak10,Peszat13}. To our best knowledge, the question of existence of càdlàg versions in the case of multiplicative noise has only been studied in \cite{Hausenblas05}. For the relation of the results of this paper to our results, see Remark~\ref{remark}. 
 
 In Section~\ref{cadlag} of this paper, we substantially generalize the aforementioned results in the case of a Lévy space--time white noise \eqref{noise}: Without any further assumptions than those required for the existence of solutions, we prove in Theorems~\ref{cadlagsolution-general}, \ref{cadlagsolution} and \ref{cadlag_bdd_domain}, for both a bounded domain $D$ and the case $D=\R^d$, that $t\mapsto u(t,\cdot)$ has a càdlàg modification in $H_r(D)$ and $H_{r,loc}(\R^d)$, respectively, for any $r<-\frac d 2$. 
%
To this end, we start our analysis by considering the stochastic heat equation on the interval $D=[0,\pi]$ in Section~\ref{bddinterval}. Treating this basic case first has the advantage that we can directly proceed to the main steps of the proof while avoiding the technical difficulties of the general case. Next, in Section~\ref{equation_whole_space}, we demonstrate how the proof for $D=[0,\pi]$ can be directly extended to the case $D=\R^d$, provided $\sigma$ is bounded and $\nu$ has finite second moments. But in order to cover the general case of Lipschitz continuous $\sigma$ and heavy-tailed noises, we need to use stopping time techniques from \cite{Chongheavytailed} to deal with the (infinitely many) large jumps of the noise, as well as results from the integration theory for general random measures (see the Appendix) to compensate the absence of finite second moments for $d\s 2$, due to the singularity of the heat kernel and the small jumps of the noise. Finally, the proof for $D=[0,\pi]$ does not extend to bounded domains in $\R^d$ with $d\s 2$ because the eigenfunctions are typically no longer uniformly bounded. Instead, the proof we give in Section~\ref{equation_bdd_domain} makes use of the fact that in the interior of $D$, the Green's function $G_D$ can be decomposed into the Gaussian density $g$ (where we can use the results of Section~\ref{equation_whole_space}) and a smooth function. With the methods reviewed in the Appendix, we also obtain sufficient control at the boundary of $D$.

Regarding the partial regularity of $t\mapsto u(t,x)$ and $x\mapsto u(t,x)$, \cite[Section~2]{loubert} obtained the following result on $D=\R^d$: If the Lévy measure $\nu$ of $L$ satisfies $\int_\R |z|^p \,\nu(\dd z)<+\infty$ for some $p<\frac 2 d$, then for fixed $t$, the process $x\mapsto u(t,x)$ has a continuous modification. Similarly, if $\int_\R |z|^p \,\nu(\dd z)<+\infty$ for some $p<1$, then there exists a continuous modification of $t\mapsto u(t,x)$ for every fixed $x$.  Extending the results of \cite{loubert}, our Theorems~\ref{continuityspacewholespace} and \ref{fixedspacewholespace}, which also apply to bounded domains, show that it suffices to check whether $\int_{[-1,1]} |z|^p \,\nu(\dd z)<+\infty$ is finite, which would include, for example, $\alpha$-stable noises with $\alpha<\frac 2 d$ (for spatial regularity) and $\alpha<1$ (for temporal regularity). Furthermore, these conditions are essentially sharp as we show in Theorems~\ref{label} and \ref{alphafixedspace}: If $\sigma\equiv 1$, and if $\nu$ has the same behavior near the origin as the Lévy measure of an $\alpha$-stable noise, then for $\frac 2 d \I \alpha < 1+\frac 2 d$ (resp. $1\I \alpha<1+\frac 2 d$), the paths of $x\mapsto u(t,x)$ (resp. $t\mapsto u(t,x)$) are unbounded on any non-empty open subset of $D$ (resp. $[0,T]$). Let us remark that the last conclusion was observed  in \cite{Mytnik03} for an $\alpha$-stable noise $\Lambda$ and the (non-Lipschitz) function $\sigma(x)=x^{\frac 1 \alpha}$ with $\alpha\in(1,1+\frac 2 d)$ via a connection between the resulting equation and stable super-Brownian motion (note that our $\alpha$ is $1+\beta$ in this reference).

In what follows, the letter $C$, occasionally with subscripts indicating the parameters that it depends on, denotes a strictly positive finite number whose value may change from line to line.

\section{Regularity of the solution in fractional Sobolev spaces}\label{cadlag}

\subsection{The stochastic heat equation on an interval}\label{bddinterval}

For the interval $D=[0,\pi]$, the Green's function $G=G_D$ has the explicit representation
\begin{equation}
	\label{greenfunction}
	G(t;x,y):=G_D(t; x,y)=\frac 2 \pi \sum_{k\s 1} \sin(kx) \sin(ky) e^{-k^2 t}\mathds 1_{t\s 0}\, .
\end{equation}
The existence and uniqueness of mild solutions to \eqref{SHE} in this case basically follow from  \cite{Chong2016}. 
\begin{prop}
	\label{existence}
	Let $\sigma\colon\R\to \R$ be a Lipschitz function, $u_0\colon[0,\pi]\to\R$ be continuous with $u_0(0)=u_0(\pi)=0$, and $L$ be a pure jump	Lévy white noise as in \eqref{noise}.
	Furthermore, define 
	\begin{equation}\label{tauN}\tau_N=\inf\left \{ t\in [0,T] \colon J\lp [0,t]\times[0,\pi]\times[-N,N]^c\rp\neq 0\right \}\, ,\quad N\in\N\,,\end{equation}
	with the convention $\inf \emptyset = +\infty$.
	Then $(\tau_N)_{N\s1}$ is an increasing sequence of stopping times such that $\tau_N>0$ and $\tau_N = +\infty$ for large values of $N$. In addition, up to modifications, \eqref{SHE} has a mild solution $u$ satisfying 
	\begin{equation}\label{p-moment}
		\sup_{(t,x)\in [0,T]\times [0,\pi]}\E\lc |u(t,x)|^p \mathds 1_{t\I \tau_N}\rc <+\infty\, ,
	\end{equation}
	for any $0<p<3$ and $N\in\N$. Furthermore, up to modifications, this solution is unique among all predictable random fields that satisfy \eqref{p-moment}.
\end{prop}
\begin{proof}
	Since $[0,\pi]$ is a bounded interval, almost surely, there is only a finite number of jumps larger than $N$ in $[0,T]\times [0,\pi]$. This immediately implies the statements about $(\tau_N)_{N\s1}$. Next, by \cite[(B.5)]{bally}, we know that $G(t;x,y)\I Cg (t,x-y)$ for any $(t,x,y)\in [0,T]\times[0,\pi]^2$, with $g$ as in \eqref{heat-kernel}. Consequently, $(1)$ to $(4)$ of Assumption B of \cite{Chong2016} are satisfied, and we can apply \cite[Theorem~3.5]{Chong2016} to obtain the existence of a unique mild solution to \eqref{SHE} satisfying \eqref{p-moment} for all $p\in(0,2]$. In order to extend this to all $p\in(2,3)$, we notice that the only step in the proof of \cite[Theorem~3.5]{Chong2016} that uses $p\I 2$ is the moment estimate (6.9) given in \cite[Lemma~6.1(2)]{Chong2016} with respect to the martingale part $L^M$. We now elaborate how this estimate can be extended to exponents $2< p<3$. For predictable processes $\phi_1$ and $\phi_2$,
	we can use \cite[Theorem~1]{marinelli} to get the upper bound
	\begin{equation}\label{firstcalc}	\begin{aligned}
		&\E\left[ \left|\int_0^t\int_{0}^\pi \int_{|z|\I N} G(t-s;x,y)\left(\sigma(\phi_1(s,y))-\sigma(\phi_2(s,y))\right) \,\tilde J (\dd s,\dd y,\dd z)\right|^p \right]\\
		&\qquad \I C\E\left[\left( \int_0^t \int_0^\pi \int_{|z|\I N} |G(t-s;x,y)|^2 |\phi_1(s,y)-\phi_2(s,y)|^2 |z|^2 \,\dd s\,\dd y \,\nu(\dd z) \right)^{\frac{p}{2}}\right]\\
		&\qquad\quad\qquad+C\E\left[ \int_0^t \int_0^\pi \int_{|z|\I N} |G(t-s;x,y)|^p |\phi_1(s,y)-\phi_2(s,y)|^p |z|^p \,\dd s\,\dd y \,\nu(\dd z) \right]\, .
	\end{aligned}\end{equation}
	Using Hölder's inequality with respect to the measure $|G(t-s;x,y)|^2\,\dd s\, \dd y$, the first term is further bounded by
	\begin{align*}
C_N\left( \int_0^t\int_{0}^\pi |G(t-s;x,y)|^2\,\dd s\,\dd y \right)^{\frac{p}{2}-1}  \int_0^t \int_0^\pi  |G(t-s;x,y)|^2 \E[|\phi_1(s,y)-\phi_2(s,y)|^p]\,\dd s\,\dd y 
	\end{align*}
	Since $G(t;x,y)\I Cg (t,x-y)$ for any $(t,x,y)\in [0,T]\times[0,\pi]^2$, and $\int_0^T \int_\R g^2(t,x) \,\dd t\,\dd x<+\infty$, we obtain the following estimate from \eqref{firstcalc} and the simple inequality $|x|^2 \I |x|+|x|^p$ for all $p\s2$:
	\begin{equation*}
		\begin{aligned}
			&\E\left[ \left|\int_0^t\int_{0}^\pi \int_{|z|\I N} G(t-s;x,y)\left(\sigma(\phi_1(s,y))-\sigma(\phi_2(s,y))\right) \,\tilde J(\dd s,\dd y,\dd z)\right|^p \right]\\
			&\qquad \I C_N\int_0^t \int_{0}^\pi (|G(t-s;x,y)|+|G(t-s;x,y)|^p)\E[|\phi_1(s,y)-\phi_2(s,y)|^p]  \,\dd s\,\dd y\,.
		\end{aligned}
	\end{equation*}
Since $\int_0^T g^p(t,x)\,\dd t\,\dd x < +\infty$ for $p<3$ (see e.g.\ \cite[(3.20)]{Chong2016}), this is exactly the extension of \cite[Lemma~6.1(2)]{Chong2016} needed to complete the proof of \cite[Theorem~3.5]{Chong2016}. 
\end{proof}



\subsubsection{The fractional Sobolev spaces $H_r([0,\pi])$}\label{sobolevspace}
For any function $f\in L^2([0,\pi])$, we can define its Fourier sine coefficients 
\begin{equation}\label{sinecoef}
 a_n(f)=\sqrt{\frac 2 \pi} \int_0^\pi f(x) \sin(nx) \,\dd x\, , \qquad n\in \N\, .
\end{equation}
Then, by Parseval's identity, $\|f\|^2_{L^2([0,\pi])}=\sum_{n\s 1}a_n(f)^2$.
For any $r\s 0$, we define 
\begin{equation*}
H_r([0,\pi]) :=\left\{ f\in L^2([0,\pi])\colon \| f \|^2_{H_r} := \sum_{n\s 1}\lp 1+n^2\rp ^r a_n(f)^2 <+\infty\right\}\, .
\end{equation*}
This is a Hilbert space for the inner product 
$\scal f h _{H_r}:= \sum_{n\s 1} \lp 1+n^2\rp ^r a_n(f) a_n(h)$.
For $r>0$, we define $H_{-r}([0,\pi])$ as the dual space of $H_{r}([0,\pi])$, that is, the space of continuous linear functionals on $H_{r}([0,\pi])$. Then $H_{-r}([0,\pi])$ is isomorphic to the space of sequences $b=(b_n)_{n\s 1}$ such that
\begin{equation*}
\| b \|^2_{H_{-r}}:= \sum_{n\s 1}\lp 1+n^2\rp ^{-r} b_n^2 <+\infty\, . 
\end{equation*}
More precisely, for $r>0$ and $\tilde f\in H_{-r}([0,\pi])$, the coefficients $b_n$ are given by $b_n= \tilde f\big(\sqrt{\frac 2 \pi}\sin(n\cdot)\big)$. Then, $\| \tilde f\|_{H_{-r}} = \| b \|_{H_{-r}}$ and the duality between $H_{-r}([0,\pi])$ and $H_{r}([0,\pi])$ is given by
$$\scal{b}{h}=\sum_{n\s 1} b_n a_n(h)\I \|b\|_{H_{-r}} \, \|h\|_{H_r}\, .$$
For example, it is easy to check that $\delta_x \in H_r([0,\pi])$ for any $x\in (0,\pi)$ and $r<-\frac 1 2$. Indeed, $\delta_x(\sin(n\cdot))=\sin(nx)$, and for any $r<-\frac 1 2$,
$$ \|\delta_x\|_{H_{r}}^2= \frac{2}{\pi}\sum_{n\s 1}\lp 1+n^2\rp ^{r} \sin^2(nx) \I \frac{2}{\pi}\sum_{n\s 1}\lp 1+n^2\rp ^{r}<+\infty\, .$$


\subsubsection{Existence of a càdlàg solution in $H_r([0,\pi])$ with $r<-\frac 1 2$}\label{cadlag-interval}
{In order to motivate why we consider fractional Sobolev spaces $H_r([0,\pi])$ with $r<-\frac 1 2$, we start with a special case.} Suppose that $b=0$ and that $\nu$ is a symmetric measure with $\nu(\R)<+\infty$. Then
we can rewrite $L=\sum_{i=1}^{N_T} Z_i \delta_{(T_i, X_i)}$, where $(T_i,X_i,Z_i)$ are the atoms of the Poisson random measure $J$, and
$$u(t,x)=\sum_{i= 1}^{N_T} G(t-T_i; x, X_i)  \sigma( u(T_i, X_i) )Z_i\, .$$
In this case, it suffices to check whether for fixed $i\s 1$, $t\mapsto G(t-T_i; \cdot, X_i)$ is càdlàg in $H_r([0,\pi])$. Using the series representation \eqref{greenfunction},
we immediately see that the function $x\mapsto G(t-T_i; x,X_i)$ belongs to $H_r([0,\pi])$ if and only if 
$$\sum_{k\s 1}(1+k^2)^r  \sin(kX_i)^2 e^{-2k^2 (t-T_i)}\mathds 1_{t\s T_i} <+\infty\, .$$
This is the case for any $r\in \R$ if $t\neq T_i$. However, for $t\downarrow T_i$, we have to restrict to $r<-\frac 1 2$. Indeed, for the càdlàg property, the only point where a problem might appear is at $t=T_i$. At this point, the existence of a left limit is obvious since $G(t-T_i; \cdot, X_i)=0$ for any $t<T_i$. For right-continuity, we use the fact that $(1-e^{-k^2 h})^2\I k^{2\eps}h^\eps$ for any $0<\eps< -\frac 1 2 -r$, so
\begin{align*}
 \left \| G(h; \cdot, X_i)-G(0; \cdot, X_i)\right \|_{H_r}^2 &=\frac 2 \pi \sum_{k\s 1}(1+k^2)^r  \sin(kX_i)^2 \lp 1-e^{-k^2 h}\rp^2\\
 &\I \frac 2 \pi \sum_{k\s 1}(1+k^2)^r  k^{2\eps}h^\eps \I C h^\eps\to 0 \qquad \text{as } h\to 0\, .
\end{align*}
Therefore, $t\mapsto u(t,\cdot)$ is càdlàg in  $H_r([0,\pi])$. 
For the general case, we first treat the drift term.
\begin{lem}\label{driftinterval}
	Assume that $u$ is the unique solution to \eqref{SHE} as in Proposition~\ref{existence}.
	{Then, 
	$$F(t,x)=V(t,x)+\int_0^t\int_0^\pi G(t-s;x,y)\sigma(u(s,y))\,\dd s\,\dd y  $$
is jointly continuous in $(t,x)\in[0,T]\times[0,\pi]$. 
	 In particular, for every $r\I 0$, the process $t\mapsto F(t,\cdot)$
	is continuous in $H_r([0,\pi])$.}
\end{lem}
\begin{proof} The continuity of $V$ is standard, and with \eqref{greenfunction}, the integral term in $F$ equals
$$ \frac{2}{\pi} \sum_{k\s 1} \sin(kx)\int_0^t\int_0^\pi e^{-k^2(t-s)} \sin(ky)\sigma(u(s,y))\,\dd s\,\dd y\, .$$
Each term in this series is jointly continuous in $(t,x)$. Hence, it suffices to show the uniform convergence of the series. Using Hölder's inequality and the fact that $u$ has uniformly bounded moments of any order $p<3$, we obtain this from 
\begin{align*}
	&\E\left[\sum_{k\s 1} \sup_{(t,x)\in[0,T]\times[0,\pi]} \left|\sin(kx) \int_0^t\int_0^\pi e^{-k^2(t-s)} \sin(ky)\sigma(u(s,y))\,\dd s\,\dd y \right|\right]\\
	&\hspace{1cm} \I C \sum_{k\s 1}\E\left[\sup_{(t,x)\in[0,T]\times[0,\pi]} \lp \int_0^t e^{-\frac53 k^2(t-s)} \, \dd s \rp^{\frac 3 5} \lp\int_0^t\int_0^\pi |\sigma(u(s,y))|^{\frac52}\,\dd s\,\dd y\rp^{\frac 2 5}\right]\\
	&\hspace{1cm} \I C\sum_{k\s 1}\lp \int_0^T e^{-\frac53 k^2s}\, \dd s\rp^{\frac 3 5}\left(\int_0^T\int_0^\pi \E[|\sigma(u(s,y))|^\frac52]\,\dd s\,\dd y  \right)^{\frac 2 5} \I C \sum_{k\s 1} {k^{-\frac65}} < +\infty\, .
\end{align*}
Then, to prove the continuity of $t\mapsto F(t,\cdot)$ in $H_r([0,\pi])$, it suffices to show the continuity in $L^2([0,\pi])$ because $L^2([0,\pi]) \hookrightarrow H_r([0,\pi])$ (that is, $L^2([0,\pi])$ is continuously embedded in $H_r([0,\pi])$). The continuity in $L^2([0,\pi])$ in turn follows from the fact that $F$ is uniformly continuous on the compact domain $[0,T]\times[0,\pi]$.
\end{proof}


\begin{prop}\label{cadlagsolutioninterval}
	Let $L$ be a pure jump Lévy white noise, and let $\sigma$ {be a bounded and Lipschitz function}. Let $u$ be the mild solution to the stochastic heat equation \eqref{SHE}. Then, for any $r<-\frac 1 2$, the stochastic process $(u(t, \cdot))_{t\in [0,T]}$ has a càdlàg version in $H_r([0,\pi])$. 
\end{prop}
\begin{proof}
	For $N\in\N$, consider the truncated noise 	
	\begin{equation}\label{LN-interval} L_N(\dd t,\dd x) = b_N\,\dd t\,\dd x + \int_{|z|\I N} z\,\tilde J(\dd t,\dd x,\dd z)\, ,\quad b_N = b-\int_{1<|z|\I N} z\,\nu(\dd z)\,,  \end{equation} 
		as well as the mild solution $u_N$ to \eqref{SHE} driven by $L_N$, that is, 
	\begin{equation}\label{truncatedmildinterval}
		u_N(t,x)=V(t,x)+\int_0^t \int_{0}^\pi G(t-s; x,y) \sigma (u_N(s,y))\, L_N(\dd s , \dd y)\,,\quad (t,x)\in[0,T]\times[0,\pi]\,.
	\end{equation}
 Then, by definition, we have $L_N = L$ and therefore also $u=u_N$ 
on the event $\{T\I \tau_N\}$, where $\tau_N$ was defined in \eqref{tauN}. Since almost surely, $\tau_N=+\infty$ for sufficiently large $N$, we have stationary convergence of the processes $u_N$ in \eqref{truncatedmildinterval} to $u$ (that is, almost surely, for large enough $N$, $u_N(t,x)=u(t,x)$ for all $(t,x)\in[0,T]\times[0,\pi]$).
	As we are interested in sample path properties of the mild solution to \eqref{SHE}, and these properties are identical to those of $u_N$ for sufficiently large $N$, it is enough to consider $u_N$ instead of $u$ in the following. The value of the parameter $N$ has no importance in our study, so we take $N=1$ for simplicity and drop the dependency in $N$. Therefore, it suffices to consider the solution to the integral equation
	\begin{equation}\label{newsolution}
		\begin{aligned}
			u(t,x)&=V(t,x)+{ b\int_0^t \int_0^\pi G(t-s;x,y) \sigma(u(s,y)) \, \dd s \, \dd y}\\
			&\quad \qquad+ \int_0^t \int_0^\pi G(t-s;x,y) \sigma(u(s,y)) \,L^M(\dd s , \dd y)\, ,
		\end{aligned}
	\end{equation}
	in other words, to assume that all jumps of $L$ are bounded by $1$.
	Furthermore, by Lemma \ref{driftinterval}, it only remains to consider the process
	\begin{equation}\label{uM} u^M(t,x)=\int_0^t \int_0^\pi G(t-s;x,y) \sigma(u(s,y)) \,L^M(\dd s , \dd y)\,. \end{equation}	
	
	For this purpose, we need to calculate the  Fourier sine coefficients defined in \eqref{sinecoef}. To lighten the notations, in what follows, we will denote these coefficients by $a_k(t)$. Then, by definition, 
	\begin{equation*}
		a_k(t)= \sqrt{\frac 2 \pi} \int_0^\pi \lp \int_{0}^t \int_0^\pi \int_{|z|\I 1} \sin(kx) G(t-s;x,y) \sigma(u(s,y))z \,\tilde J( \dd s, \dd y , \dd z) \rp \,\dd x\, ,\quad k\s 1\,.
	\end{equation*}
	We want to exchange the stochastic integral and the Lebesgue integral, and because all involved terms are square-integrable, Theorem~\ref{fubini} with $p=2$ allows us to do so. 
	Therefore, 
	\begin{equation}
		\begin{aligned} \label{ak}
			a_k(t)&=  \int_{0}^t \int_0^\pi \int_{|z|\I 1} \sqrt{\frac 2 \pi} \lp  \int_0^\pi \sin(kx) G(t-s;x,y) \,\dd x \rp  \sigma(u(s,y))z  \,\tilde J(\dd s, \dd y , \dd z)\\
			&= \sqrt{\frac 2 \pi} e^{-k^2t}  \int_{0}^t \int_0^\pi   \sigma(u(s,y))  \sin(ky) e^{k^2 s}  \,L^M(\dd s, \dd y )\, .
		\end{aligned}
	\end{equation}
	
	We gather some moment estimates for the family of integrals
	\begin{equation}\label{int}
		I_a^b(k):=\int_a^b \int_0^\pi \sin(ky) e^{k^2 s} \sigma(u(s,y)) \,L^M(\dd s , \dd y)\,,
	\end{equation}
	where $0 \I a < b\I T$.
	Since $\sigma$ is bounded, we can estimate the second and fourth moments of $I_a^b(k)$ using \cite[Theorem~1]{marinelli}:
	\begin{equation}\label{moment2}
		\begin{aligned}
			\E\lc I_a^b(k)^2\rc &\I C \int_a^b \int_0^\pi \sin^2 (ky) e^{2k^2s} \E\lc \left |\sigma(u(s,y)) \right |^2 \rc \,\dd s \,\dd y \\
			& \I  C \int_a^b e^{2k^2s} \,\dd s= C\frac{ e^{2k^2b}- e^{2k^2a}}{2k^2}\, ,\\
			\E\lc I_a^b(k)^4\rc & \I C \lp  \E\lc \lp \int_a^b \int_0^\pi \sin^2 (ky) e^{2k^2s} \left |\sigma(u(s,y)) \right |^2  \,\dd s\, \dd y\rp^2 \rc \right. \\
			&\quad\qquad\left .\vphantom{\lp \int \rp^2} + \int_a^b \int_0^\pi \sin^4 (ky) e^{4k^2s} \E\lc \left |\sigma(u(s,y)) \right |^4 \rc\, \dd s \,\dd y \rp \\
			&\I  C\left(\frac{ (e^{2k^2b}- e^{2k^2a})^2}{4k^4} +\frac{ e^{4k^2b}- e^{4k^2a}}{4k^2}\right)  \, ,
		\end{aligned}
	\end{equation}
	where $C$ also depends on $\int_{|z|\I 1} z^2\, \nu(\dd z)$ and $\int_{|z|\I 1} z^4\, \nu(\dd z)$, both of which are finite. 
	
	Also, for $0\I a<b\I c <d \I T$, still assuming that $\sigma$ is bounded, 
	\begin{equation}\label{productformula}
		\begin{aligned}
			\E\lc \lp I_a^b(k) I_c^d(j)\rp^2 \rc &= \E\lc \lp  \int_c^d \int_0^\pi I_a^b(k) \sin(jy) e^{j^2 s} \sigma(u(s,y)) \,L^M(\dd s , \dd y)\rp^2 \rc \\
			&\I C \int_c^d \int_0^\pi \E \lc I_a^b(k) ^2  \sigma(u(s,y))^2 \rc  \sin^2(jy) e^{2j^2 s} \,\dd s \,\dd y\\
			&\I C\lp \int_a^b e^{2k^2 s}  \,\dd s\rp \lp \int_c^d e^{2j^2 s}  \,\dd s \rp = C \frac{e^{2k^2b}-e^{2k^2 a}}{2k^2} \frac{e^{2j^2d}-e^{2j^2 c}}{2j^2} \, ,
		\end{aligned}
	\end{equation}
	where we used the fact that $I_a^b(k)$ is $\m F_c$-measurable, and \eqref{moment2} in the second inequality. 
	
	We will use \cite[Chapter III, \S4, Theorem 1]{skorohod} to show the existence of a càdlàg version of $t\mapsto u^M(t,\cdot)$. By \cite[Chapter III, \S2, Theorem~1]{skorohod}, $u$ has a separable version, which is, because of \cite[Theorem~4.7]{Chong2016} and \cite[Lemma~B.1]{bally}, continuous in $L^2(\Omega)$, and therefore $t\mapsto u^M(t,\cdot)$ is continuous in $L^2(\Omega)$ as a process with values in $L^2([0,\pi])$ (and thus in $H_r([0,\pi])$ since $r<-\frac 1 2$). Then it suffices to show that for any $t\in [0,T]$, $u(t, \cdot) \in H_r([0,\pi])$, and that for some $\delta>0$, 
	\begin{equation*}
		\E\lc \|u^M(t+h, \cdot)-u^M(t,\cdot)\|_{H_r}^2 \| u^M(t-h, \cdot)-u^M(t,\cdot)\|_{H_r}^2 \rc \I C h^{1+\delta}
	\end{equation*} 
	for any $h\in(0,1)$. 	By \eqref{moment2}, we have
	$\E\lc a_k^2(t) \rc\I CT$, so for $r<-\frac12$,
	we have $$\sum_{k\s 1} (1+k^2)^r a_k^2(t) <+\infty$$ almost surely, and $u(t,\cdot)\in H_r([0,\pi])$ is proved. Next,
	$$\|u^M(t\pm h, \cdot)-u^M(t,\cdot)\|_{H_r}^2=\sum_{k\s 1} (1+k^2)^r \lp a_k(t\pm h)-a_k(t)\rp^2\,, $$
so 
	using \eqref{ak},
	\begin{equation*}
		\begin{aligned}
			&a_k(t+h)-a_k(t)=-\sqrt{\frac 2 \pi} e^{-k^2t} \lc (1-e^{-k^2h})I_0^t(k)-e^{-k^2h} I_t^{t+h}(k)  \rc\, , \\
			&a_j(t-h)-a_j(t)=-\sqrt{\frac 2 \pi}e^{-j^2(t-h)} \lc (e^{-j^2h}-1)I_0^{t-h}(j)+ e^{-j^2h} I_{t-h}^{t}(j)  \rc\, .
		\end{aligned}
	\end{equation*}
	Therefore, using the classical inequality $(a+b)^2\I 2 (a^2+b^2)$,
	\begin{equation}\label{star}
		\begin{aligned}
			&\|u^M(t+h, \cdot)-u^M(t,\cdot)\|_{H_r}^2  \| u^M(t-h, \cdot)-u^M(t,\cdot)\|_{H_r}^2 \\
			&\qquad \I \frac{4}{\pi^2} \sum_{k,j \s 1} (1+k^2)^r(1+j^2)^r \lp |A_1(j,k)|+|A_2(j,k)|+|A_3(j,k)|+|A_4(j,k)| \rp^2\\
			&\qquad\I C \sum_{k,j \s 1} (1+k^2)^r(1+j^2)^r \lp A_1(j,k)^2+A_2(j,k)^2+A_3(j,k)^2+A_4(j,k)^2 \rp\, ,
		\end{aligned}
	\end{equation}
	for some constant $C$, where 
	\begin{equation*}
		\begin{aligned}
			A_1(j,k)&:= e^{-k^2t}e^{-j^2(t-h)}(1-e^{-k^2h})(e^{-j^2h}-1)I_0^t(k)I_0^{t-h}(j) \, ,\\
			A_2(j,k)&:= e^{-k^2t}e^{-j^2(t-h)}(1-e^{-k^2h})e^{-j^2h}I_0^t(k)I_{t-h}^t(j)\, ,\\
			A_3(j,k)&:= e^{-k^2t}e^{-j^2(t-h)}e^{-k^2h}(e^{-j^2h}-1)I_t^{t+h}(k)I_0^{t-h}(j)\, ,\\
			A_4(j,k)&:= e^{-k^2t}e^{-j^2(t-h)}e^{-k^2h}e^{-j^2h}I_t^{t+h}(k)I_{t-h}^t(j)\, .
		\end{aligned}
	\end{equation*}
	We treat each of the four terms separately. 
	
\vspace{0.25\baselineskip}
	
	\noindent
	\underline{$A_1(j,k)$:} 
	\begin{equation*}
		\begin{aligned}
			\E\lc A_1(j,k)^2 \rc &=e^{-2k^2t}e^{-2j^2(t-h)}(1-e^{-k^2h})^2(e^{-j^2h}-1)^2\E\lc ( I_0^t(k)I_0^{t-h}(j)) ^2 \rc \\
			& \I C e^{-2k^2t}e^{-2j^2(t-h)}(1-e^{-k^2h})^2(e^{-j^2h}-1)^2 \\
			& \quad\qquad \times \E\lc ( I_0^{t-h}(k)I_0^{t-h}(j)) ^2 + ( I_{t-h}^t(k)I_0^{t-h}(j)) ^2  \rc\\
			&=: \tilde A_1(j,k) + \tilde A_2(j,k)  \, .
		\end{aligned}
	\end{equation*}
	By \eqref{productformula}, we can write
	\begin{equation*}
		\begin{aligned}
			\E\lc  ( I_{t-h}^t(k)I_0^{t-h}(j)) ^2  \rc &\I C\, \frac{ e^{2j^2(t-h)}-1 }{2j^2}\, \frac{ e^{2k^2t}-e^{2k^2(t-h)}}{2k^2} \\
			&\I  Ce^{2j^2(t-h)}\, \frac{1-e^{-2j^2(t-h)}}{2j^2} e^{2k^2t}\,  \frac{1-e^{-2k^2h}}{2k^2} &\I C\, \frac{e^{2k^2t}e^{2j^2(t-h)}}{2j^2}\, h\, ,
		\end{aligned}
	\end{equation*}
	where we used $1-e^{-2k^2h} \I 2k^2h$ and $1-e^{-2j^2(t-h)} \I 1$ in the last inequality. Since $(1-e^{-k^2h})^2\I 1$ and $(1-e^{-j^2h})^2\I j^2 h$,  we deduce that
	\begin{equation}\label{A2}
		\tilde A_2(j,k)\I C h^2\, .
	\end{equation}
	Also, by the Cauchy-Schwarz inequality,
	\begin{equation}\label{B2}
		\begin{aligned}
			\E\lc  ( I_{0}^{t-h}(k)I_0^{t-h}(j)) ^2  \rc &\I  \E\lc  I_{0}^{t-h}(k)^4 \rc^{\frac 1 2}  \E\lc  I_{0}^{t-h}(j)^4 \rc^{\frac 1 2}\, .
		\end{aligned}
	\end{equation}
	By \eqref{moment2} and subadditivity of the square root,
	\begin{equation}\label{order4}
		\begin{aligned}
			&\E\lc  I_{0}^{t-h}(k)^4 \rc^{\frac 1 2}
			\I C\Bigg(\frac{e^{2k^2(t-h)}-1}{2k^2}+ \Bigg( \frac{e^{4k^2(t-h)}-1}{4k^2} \Bigg)^{\frac 1 2} \Bigg)\\
			&\qquad\I C e^{2k^2t} \Bigg(\frac{e^{-2k^2h}-e^{-2k^2t}}{2k^2}+ \Bigg( \frac{e^{-4k^2h}-e^{-4k^2t}}{4k^2} \Bigg)^{\frac 1 2} \Bigg) \I  C e^{2k^2t} \lp\frac{1}{2k^2}+  \frac{1}{2k}  \rp  \, .
		\end{aligned}
	\end{equation}
	Let $0<\delta<\frac 3 2$, to be chosen later. Then, multiplying each term by $(1-e^{-k^2h})^2$ and using $(1-e^{-k^2h})^2\I  k^{2} h$ for the first term, and
$(1-e^{-k^2h})^2=(1-e^{-k^2h})^{\frac 1 2 +\delta}(1-e^{-k^2h})^{\frac 3 2 -\delta} \I  k^{1+2\delta} h^{\frac 1 2+\delta}$
	for the second term of the sum, we get
	\begin{equation}\label{B3}
		\begin{aligned}
			(1-e^{-k^2h})^2  \E\lc  I_{0}^{t-h}(k)^4 \rc^{\frac 1 2}&\I C e^{2k^2t} \lp h + k^{2\delta} h^{\frac 1 2 +\delta}  \rp\, .
		\end{aligned}
	\end{equation}
	A similar calculation yields 
	\begin{equation}\label{B4}
		\begin{aligned}
			(1-e^{-j^2h})^2\E\lc  I_{0}^{t-h}(j)^4 \rc^{\frac 1 2} \I C e^{2j^2(t-h)} \lp h + j^{2\delta} h^{\frac 1 2 +\delta}  \rp\, .
		\end{aligned}
	\end{equation}
	Then, we combine \eqref{B2}, \eqref{B3} and \eqref{B4} to obtain 
	\begin{equation}\label{A3}
		\tilde A_1(j,k)\I C \lp h^2+j^{2\delta}k^{2\delta} h^{1+2\delta}\rp\, .
	\end{equation}
	Therefore, \eqref{A2} and \eqref{A3} give
$$\E\lc A_1(j,k)^2 \rc \I C\lp h^2+ j^{2\delta}k^{2\delta} h^{1+2\delta}\rp\,.$$
	
\vspace{0.25\baselineskip}
	
	\noindent
	\underline{$A_2(j,k)$:} We treat this term in a similar way to $A_1(j,k)$:
	\begin{equation*}
		\begin{aligned}
			\E\lc A_2(j,k)^2\rc &= e^{-2k^2t}e^{-2j^2(t-h)}(1-e^{-k^2h})^2e^{-2j^2h}\E\lc ( I_0^t(k)I_{t-h}^t(j)) ^2 \rc\\
			&\I C e^{-2k^2t}e^{-2j^2(t-h)}(1-e^{-k^2h})^2e^{-2j^2h}\E\lc (  I_0^{t-h}(k)I_{t-h}^{t}(j)) ^2+ ( I_{t-h}^t(k)I_{t-h}^t(j)) ^2 \rc \\
			&=: B_1(j,k)+ B_2(j,k)\, .
		\end{aligned}
	\end{equation*}
	In the same way as for the term $\tilde A_2(j,k)$, we get
	\begin{equation}\label{indepterm}
		B_1(j,k)\I Ch^2 \, .
	\end{equation}
	We use the Cauchy-Schwarz inequality to deal with the term $B_2(j,k)$:
	\begin{equation*}
		\begin{aligned}
			B_2(j,k)\I C e^{-2k^2t}e^{-2j^2(t-h)}(1-e^{-k^2h})^2e^{-2j^2h} \E\lc  I_{t-h}^{t}(k)^4 \rc^{\frac 1 2}  \E\lc  I_{t-h}^{t}(j)^4 \rc^{\frac 1 2}\, .
		\end{aligned}
	\end{equation*}
	As in \eqref{order4}, we get
	\begin{align*}
		\E\lc  I_{t-h}^{t}(j)^4 \rc^{\frac 1 2} &\I C e^{2j^2t} \Bigg(\frac{1-e^{-2j^2h}}{2j^2}+ \Bigg( \frac{1-e^{-4j^2h}}{4j^2} \Bigg)^{\frac 1 2} \Bigg) \I C e^{2j^2t} \lp h+ \sqrt h \rp\, ,
	\end{align*}
and similarly $\E\lc  I_{t-h}^{t}(k)^4 \rc^{\frac 1 2} \I C e^{2k^2t} \lp h+ \sqrt h \rp$.
	Also, for $0<\delta <1$, since $(1-e^{-k^2h})^2\I  k^{2\delta}h^\delta$, 
	\begin{equation}\label{bterm2}
		B_2(j,k)\I C k^{2\delta} h^{\delta} \lp h+ \sqrt h \rp^2 \I C k^{2\delta} h^{1+\delta}\, .
	\end{equation}
	By \eqref{indepterm} and \eqref{bterm2},
	\begin{equation*}
		\E\lc A_2(j,k)^2\rc \I C k^{2\delta}h^{1+\delta} \, .
	\end{equation*}
	
\vspace{0.25\baselineskip}
	
	\noindent
	\underline{$A_3(j,k)$:} By \eqref{productformula},
	\begin{equation*}
		\begin{aligned}
			\E\lc A_3(j,k)^2\rc &=  e^{-2k^2t}e^{-2j^2(t-h)}e^{-2k^2h}(e^{-j^2h}-1)^2 \E \lc ( I_t^{t+h}(k)I_0^{t-h}(j)) ^2 \rc \\
			&\I C e^{-2k^2t}e^{-2j^2(t-h)}e^{-2k^2h}(e^{-j^2h}-1)^2 \, \frac{e^{2j^2(t-h)}-1}{2j^2} \, \frac{e^{2k^2(t+h)}-e^{2k^2t}}{2k^2} \\
			&\I C \, (e^{-j^2h}-1)^2 \, \frac{1-e^{-2j^2(t-h)}}{2j^2} \, \frac{1-e^{-2k^2h}}{2k^2} \I C \, \frac{(e^{-j^2h}-1)^2}{2j^2}\, \frac{1-e^{-2k^2h}}{2k^2}  \, .
		\end{aligned}
	\end{equation*}
	Then, since $(1-e^{-j^2h})^2 \I j^2h$ and $1-e^{-2k^2h} \I 2k^2h$, we get
	\begin{equation*}
		\E\lc A_3(j,k)^2\rc \I C h^2 \, .
	\end{equation*} 
	
\vspace{0.25\baselineskip}
	
	\noindent
	\underline{$A_4(j,k)$:} Again, by \eqref{productformula},
	\begin{equation*}
		\begin{aligned}
			\E\lc A_4(j,k)^2\rc &= e^{-2k^2t}e^{-2j^2(t-h)}e^{-2k^2h}e^{-2j^2h}\E \lc ( I_t^{t+h}(k)I_{t-h}^t(j))^2\rc \\
			&\I C e^{-2k^2t}e^{-2j^2(t-h)}e^{-2k^2h}e^{-2j^2h}\, \frac{e^{2j^2t}-e^{2j^2(t-h)}}{2j^2}\, \frac{e^{2k^2(t+h)}-e^{2k^2t}}{2k^2} \\
			&\I C\, \frac{1-e^{-2j^2h}}{2j^2}\, \frac{1-e^{-2k^2h}}{2k^2}\, .
		\end{aligned}
	\end{equation*}
	Therefore, as for the previous term we get
	\begin{equation*}
		\E\lc A_4(j,k)^2\rc \I C h^2 \, .
	\end{equation*} 
	Then, for every $r<-\frac 1 2$, we can pick $0<\delta<1$ such that $r+\delta<-\frac 1 2$. Then,
	\begin{equation*}
		\E\lc \|u^M(t+h, \cdot)-u^M(t,\cdot)\|_{H_r}^2 \| u^M(t-h, \cdot)-u^M(t,\cdot)\|_{H_r}^2\rc \I C h^{1+\delta}\, ,
	\end{equation*}
	so we deduce that $(u^M(t,\cdot))_{t\s 0}$ has a càdlàg version in $H_r([0,\pi])$ for any $r<-\frac 1 2$.
\end{proof}

\begin{rem}\label{extensionrk}
 The result of Proposition \ref{cadlagsolutioninterval} is in fact valid for any predictable random field $u$ whose Fourier sine coeficients can be written in the form
 $$a_k(u(t,\cdot)) = C  e^{-k^2t} \int_0^t\int_0^\pi \sin(ky) e^{k^2s} Z(s,y)\,  {L(\dd s , \dd y)} \, ,$$
where $Z$ is another predictable and bounded random field.
\end{rem}

For unbounded $\sigma$, we deduce the result from Proposition \ref{cadlagsolutioninterval} via an approximation argument.
\begin{theo}\label{cadlagsolution-general}
Let $u$ be the mild solution to the stochastic heat equation \eqref{SHE} constructed in Proposition~\ref{existence}. Then, for any $r<-\frac 1 2$, the  process $(u(t, \cdot))_{t\in [0,T]}$ has a càdlàg version in $H_r([0,\pi])$. 
\end{theo}

\begin{rem} The constraint $r<-\frac12$ in Theorem~\ref{cadlagsolution-general} is optimal. This follows from the discussion at the beginning of Section~\ref{cadlag-interval} and the fact that the Dirac delta distribution $\delta_a$, $a\in(0,\pi)$, does not belong to $H_s([0,\pi])$ for any $s\s-\frac12$.
\end{rem}

\begin{proof}[Proof of Theorem \ref{cadlagsolution-general}]
By the argument given at the beginning of the proof of Proposition~\ref{cadlagsolutioninterval} and by Lemma \ref{driftinterval}, we only need to consider $u^M$ as defined in \eqref{uM}. Let $\sigma_n(u)= \sigma(u)\mathds 1_{|u|\I n}$. We define
 $$u^M_n(t,x)=\int_0^t \int_0^\pi G(t-s;x,y) \sigma_n(u(s,y)) \,L^M(\dd s , \dd y) \, .$$
As in \eqref{ak}, the Fourier sine coefficients of $t\mapsto u^M(t, \cdot)-u^M_n(t,\cdot)$ are given by
%
\begin{align*}
 a_{k,n}(t)=\sqrt{\frac 2 \pi} \int_0^t \int_0^\pi \sin(ky) e^{-k^2(t-s)}(  \sigma(u(s,y))- \sigma_n(u(s,y))  )\, L^M(\dd s, \dd y) \, .
\end{align*}
 Therefore, for any $t\in [0,T]$,
\begin{align}\label{hsnorm2interval}
 \| u^M(t,\cdot)-u^M_n(t,\cdot)\|_{H_r}^2=\sum_{k\s 1} (1+k^2)^r  a_{k,n}^2(t)\, .
\end{align}
Then, using
$e^{-k^2(t-s)}=1-\int_s^t k^2 e^{-k^2(t-r)}\,\dd r$
and Theorem~\ref{fubini} and \eqref{p-est} with $p=2$, we can rewrite 
\begin{align*}
 a_{k,n}(t) &=\sqrt{\frac 2 \pi}\lp \int_0^t \int_0^\pi \sin(ky) \sigma_{(n)} (s,y) \,L^M(\dd s, \dd y)\right . \\
 &\quad \qquad\left . - \int_0^t \int_0^\pi \sin(ky) \lp \int_s^t k^2 e^{-k^2(t-r)}\,\dd r\rp \sigma_{(n)} (s,y)\,L^M(\dd s, \dd y) \rp \\
 &=\sqrt{\frac 2 \pi}\lp \int_0^t \int_0^\pi \sin(ky) \sigma_{(n)} (s,y)\, L^M(\dd s, \dd y)\right . \\
 &\quad \qquad\left . - \int_0^t  k^2 e^{-k^2(t-r)} \lp \int_0^r \int_0^\pi \sin(ky) \sigma_{(n)} (s,y)\,L^M(\dd s, \dd y) \rp \, \dd r \rp \, ,
\end{align*}
where $\sigma_{(n)} (s,y):= \sigma(u(s,y))- \sigma_n(u(s,y))$. Therefore,
\begin{equation}\label{calc-akn}
\begin{aligned}
\left | a_{k,n}(t) \right | 
&\I C \sup_{r \in [0,t]} \left | \int_0^r \int_0^\pi \sin(ky) \sigma_{(n)} (s,y)\, L^M(\dd s, \dd y) \right |   \lp 1+ \int_0^t k^2 e^{-k^2(t-r)}\,\dd r \rp \\
&\I C \sup_{r \in [0,t]} \left | \int_0^r \int_0^\pi \sin(ky) \sigma_{(n)} (s,y)\,L^M(\dd s, \dd y) \right |  \, ,
\end{aligned}
\end{equation}
where $C$ does not depend on $k$. So by Doob's inequality, we deduce that
\begin{equation}\label{expectsupinterval}
 \E \lc \sup_{t\in [0,T]} a_{k,n}^2(t) \rc \I C \int_0^T \int_0^\pi \sin^2(ky) \E \lc \sigma_{(n)}^2 (s,y) \rc \dd s \,\dd y\, .
\end{equation}
By \eqref{p-moment}, \eqref{hsnorm2interval} and \eqref{expectsupinterval}, it follows from dominated convergence that for any $r<-\frac 1 2$,
\begin{equation*}
	\begin{aligned}
	 \E \lc \sup_{t\in [0,T]}  \| u(t,\cdot)-u_n(t,\cdot)\|_{H_r}^2\rc &\I C \sum_{k\s 1} (1+k^2)^r    \int_0^T \int_0^\pi \sin^2(ky) \E \lc \sigma_{(n)}^2 (s,y) \rc \dd s \,\dd y\to 0
	\end{aligned}
\end{equation*}
as $n\to+\infty$.
Therefore, $ \sup_{t\in [0,T]}  \| u(t,\cdot)-u_n(t,\cdot)\|_{H_r}\to 0$ in $L^2(\Omega)$ as $n\to +\infty$, and there is a subsequence $(n_k)_{k\s 0}$ such that $\sup_{t\in [0,T]}  \| u(t,\cdot)-u_{n_k}(t,\cdot)\|_{H_r}\to 0$ almost surely as $k\to +\infty$. This means that $u_{n_k}(t,\cdot)$ converges to $u(t,\cdot)$ in $H_r([0,\pi])$ uniformly in time for any $r<-\frac 1 2$. Since $\sigma_{n_k}$ is bounded, $t \mapsto u_{n_k}(t,\cdot)$ has  a càdlàg version in $H_r([0,\pi])$ by Proposition \ref{cadlagsolutioninterval} and Remark~\ref{extensionrk}. Therefore, $t \mapsto u(t,\cdot)$ has a càdlàg version in $H_r([0,\pi])$ for any $r<-\frac 1 2$. 
\end{proof}

\subsection{The stochastic heat equation on $\R^d$}\label{equation_whole_space}
In \cite{Chongheavytailed}, the first author proved the existence of a solution to the stochastic heat equation on $\R^d$ under assumptions on the driving noise that are general enough to include the case of $\alpha$-stable noises. More specifically, suppose that $D=\R^d$ in \eqref{SHE} and
 that the following hypotheses hold:
\begin{itemize}
 \item[\hypertarget{hyp1}{\textbf{(H)}}] There exists $0<p<1+\frac 2 d$ and $\frac{p}{1+\lp 1+\frac 2 d -p\rp}<q\I p$ such that
\begin{equation*}
 \int_{|z|\I 1}|z| ^p \,\nu(\dd z)+\int_{|z|>1}|z|^q\,\nu( \dd z)<+\infty\, .
\end{equation*}
If $p<1$, we assume that
$b_0:=b-\int_{|z|\I 1} z\,\nu(\dd z)=0$.
\end{itemize}
In contrast to the situation on a bounded domain, we can no longer use the stopping times $\tau_N$ in \eqref{tauN} (with $[0,\pi]$ replaced by $\R^d$) to localize equation \eqref{SHE}. Indeed, since $\R^d$ is unbounded, the $\dd t\,\dd x\,\nu(\dd z)$-measure of $[0,T]\times\R^d \times [-N,N]^c$ will in general be infinite for any $N\in\N$. In particular, on any time interval $[0,\eps]$ where $\eps>0$, we already have infinitely many jumps of arbitrarily large size, which implies that $\tau_N=0$ almost surely for all $N\in\N$. 

Therefore,  instead of using the stopping times \eqref{tauN}, the idea is to use truncation levels that increase with the distance to the origin. More precisely, let $h\colon\R^d\to \R$ be the function $h(x)=1+|x|^\eta$, for some $\eta$ to be chosen later, and define for $N\in \N$,
\begin{equation}\label{stoppingtimes}
\tau_N= \inf \left \{t\in [0,T] \colon J\lp [0,t] \times \left \{ (x,z)\colon |z| > Nh(x)\right \} \rp>0 \right \}\, .
\end{equation}
For every $N\s 1$, we can now introduce a truncation of $L$ by 
 \begin{equation}\label{LN}
 L_N( \dd t, \dd x)= b\,\dd t\,\dd x+\int_{|z|\I 1} z\, \tilde J(\dd t, \dd x, \dd z)+\int_{1<|z|<Nh(x)} z  \,J(\dd t, \dd x, \dd z)
\end{equation}
(do not confuse this with $L_N$ defined in \eqref{LN-interval}, which was for the case of the interval $[0,\pi]$),
which in turn gives rise to the equation
\begin{equation}\label{truncatedmild}
 u_N(t,x)= V(t,x)+ \int_0^t \int_{\R^d} g(t-s, x-y) \sigma (u_N(s,y)) L_N(\dd s , \dd y)\,.
\end{equation}

\begin{prop}\label{existencerealline}
Let $\sigma$ be Lipschitz continuous, $u_0$ be bounded and continuous, and $L$ be a Lévy white noise as in \eqref{noise} satisfying \hyperlink{hyp1}{\textbf{(H)}} for some $p,q>0$. Then, if we choose $\frac d q<\eta<\frac{2-d(p-1)}{p-q}$, we have $\tau_N>0$ for every $N\s 1$ and almost surely, $\tau_N= +\infty$ for large $N$ (recall the convention $\inf \emptyset = +\infty$). Moreover, for any $N\s 1$, there exists a solution $u_N$ to \eqref{truncatedmild} such that for some constant $C_N<+\infty$, we have
\begin{equation}\label{moment1}
 \sup_{(t,x)\in [0,T]\times [-R,R]^d}\E \lc\left |u_N(t,x) \right | ^p \rc \I  E_{\frac{2-d (p-1)}{2(p\vee 1)},\frac{1}{p\vee 1}}\Big( C_N R^{\frac{\eta(p-q)}{p\vee1}}\Big) < + \infty\, ,\quad R>1\, ,
\end{equation}
where $E_{\alpha,\beta}(z)=\sum_{k \s 0} \frac{z^k}{\Gamma(\alpha k+\beta)}$ for $\alpha,\beta > 0$ and $z\in\R$ are the Mittag--Leffler functions. 

Furthermore, we have $u_N(t,x) = u_{N+1}(t,x)$ on $\{t \I \tau_N \}$, and the random field $u$ defined by $u(t,x)=u_N(t,x)$ on $\{t\I \tau_n\}$, is a mild solution to \eqref{SHE} on $D=\R^d$.
 \end{prop}
\begin{proof}
	The result is a direct application of \cite[Theorem 3.1]{Chongheavytailed} except for the moment property \eqref{moment1}. The finiteness of the left-hand side is included in the cited theorem for $p< 1+\frac 2 d$. In the case $d=1$ and $2<p<3$, the only thing we need is an extension of \cite[Lemma~3.3(2)]{Chongheavytailed}, which can be obtained by combining the arguments given in the proof of \cite[Lemma~3.3(2)]{Chongheavytailed} and the proof of Proposition~\ref{existence}. Indeed, for predictable $\phi_1$ and $\phi_2$, proceeding as in the proof of \cite[Lemma~3.3]{Chongheavytailed} but using the moment inequalities of \cite[Theorem~1]{marinelli}, we see that
	\begin{align*}
		&\E\left[ \left|\int_0^t\int_{\R} g(t-s,x-y)\left(\sigma(\phi_1(s,y))-\sigma(\phi_2(s,y))\right) L_N(\dd s,\dd y)\right|^p \right]\\
		&\qquad \I C\int_0^t \int_\R (g(t-s,x-y)+g^p(t-s,x-y))\E[|\phi_1(s,y)-\phi_2(s,y)|^p] h(y)^{p-q}\dd s\,\dd y\, .
	\end{align*}

	In order to obtain the bound involving the Mittag--Leffler functions, observe from the calculations between  the last display on page 2272 and equation (3.13) of \cite{Chongheavytailed} that for every fixed $N\s1$, there exists $C_N<+\infty$ such that
	\begin{align*} &\sup_{(t,x)\in [0,T]\times [-R,R]^d}\E[|u_N(t,x)|^p] \I \sum_{n\s 1} C_N^n\left(  \frac{\left( R^{n\eta(p-q)}+\Gamma\left(\textstyle\frac{1+{n\eta(p-q)}}{2}\right)\right)\Gamma\left(1-\frac{d}{2}(p-1)\right)^n}{\Gamma\left(1+(1-\frac{d}{2}(p-1))n\right)} \right)^{\frac 1 {p\vee1}}\\
		&\qquad \I \sum_{n\s 1} C_N^n \left(  \frac{\Gamma\left(\textstyle\frac{1+{n\eta(p-q)}}{2}\right)}{\Gamma\left(1+(1-\frac{d}{2}(p-1))n\right)}\right)^{\frac 1 {p\vee1}} + \sum_{n\s 1} C_N^n \left(  \frac{R^{n\eta(p-q)}}{\Gamma\left(1+(1-\frac{d}{2}(p-1))n\right)}\right)^{\frac 1 {p\vee1}}\,. \end{align*}
	The first series converges for our choice of $\eta$. Furthermore, for all $a>0$, we have by Stirling's formula that $\Gamma(1+ax)^\frac{1}{p\vee1} \s C \Gamma(\frac{1+ax}{p\vee1})$ when $x>1$. Hence,
	\[ \sup_{(t,x)\in [0,T]\times [-R,R]^d}\E[|u_N(t,x)|^p] \I C_N + \sum_{n\s 1} \frac{\big(C_N R^{\frac{\eta(p-q)}{p\vee1}}\big)^n}{\Gamma\left(\frac{1}{p\vee1} +\frac{(2-d(p-1))n}{2(p\vee1)} \right)} \I E_{\frac{2-d (p-1)}{2(p\vee 1)},\frac{1}{p\vee 1}}\Big( C_N R^{\frac{\eta(p-q)}{p\vee1}}\Big)\,,  \]
	which is \eqref{moment1}.
\end{proof}


 \subsubsection{Stationarity of the solution}

The proof of Proposition~\ref{existencerealline} heavily relies on the stopping times $\tau_N$ introduced in \eqref{stoppingtimes}. These are ``centered'' around the origin in the sense that large jumps are permitted if they occur far enough from $x=0$. As a consequence, even if the initial condition is constant, it does not follow a priori from Proposition~\ref{existencerealline} that the solution $u$ to \eqref{SHE} is stationary in space. On the other hand, of course, choosing to center $\tau_N$ around the origin is completely arbitrary. So in this section, we show that the solution constructed in Proposition~\ref{existencerealline} remains the same if we take other spatial reference points for $\tau_N$, from which the stationarity of the solution in space will follow.

To this end, let $\frac d q<\eta<\frac{2-d(p-1)}{p-q}$ and define the family of stopping times $\tau_N^a$ by
\begin{equation*}
 \tau_N^a:=\inf\left \{ t\in [0,T]\colon  J\lp [0,t]\times\left \{(x,z)\colon |z|>Nh(x-a)  \right \} \rp >0 \right \}\,,\quad N\in\N\,,\quad a\in\R^d\,.
\end{equation*}
In particular, $\tau_N^0$ is the same as $\tau_N$ defined in \eqref{stoppingtimes}. Since the intensity measure of $J$ is invariant under translation in the space variable, $\tau_N^a$ has the same law as $\tau_N$, and the conclusions of Proposition~\ref{existencerealline} are valid for $\tau_N^a$. In particular, for any $N\s1$, almost surely $\tau_N^a>0$, and $\tau_N^a= +\infty$ for large $N$. Furthermore, by definition, on the event $\left \{ t\I \tau_N^a\right \}$, $L(\dd t , \dd x)=L_N^a(\dd t , \dd x)$, where
\begin{equation*}
L_N^a(\dd t , \dd x):= b\, \dd t \,\dd x + \int_{|z|\I 1} z \, \tilde J(\dd t , \dd x, \dd z) +\int_{1<|z| \I Nh(x-a)} z \, J(\dd t , \dd x, \dd z) \, .
\end{equation*}
\begin{prop}\label{lem1}
Let $\sigma \colon \R\to \R$ be Lipschitz continuous, $u_0$ be bounded and continuous, and $L$ be a Lévy white noise as in \eqref{noise} fulfilling the assumption \hyperlink{hyp1}{\textbf{(H)}} with $p,q>0$. Then for any $N\in \N$ and $a\in\R^d$, there exists a mild solution $u_N^a$ to \eqref{SHE} with noise $L^a_N$ instead of $L$ such that \eqref{moment1} also holds for $u_N^a$.
 Moreover, for $a,b\in \R^d$, $N\in \N$ and $(t,x) \in [0,T]\times \R^d$, we have
\begin{equation}\label{asequality}
 u_N^a(t,x)\mathds 1_{t\I \tau_N^a\wedge \tau_N^b}=u_N^b(t,x)\mathds 1_{t\I \tau_N^a\wedge \tau_N^b}  \quad \text{a.s.}
\end{equation}
\end{prop}
\begin{proof}
The first part is proved in the same way as Proposition~\ref{existencerealline}. For \eqref{asequality}, we observe that $L_N^a=L_N^b$ on  $\{t< \tau_N^a\wedge \tau_N^b\}$. Then we use the construction of the solutions $u_N^a$ and $u_N^b$ via a Picard iteration scheme as in the proof of \cite[Theorem~3.1]{Chongheavytailed} and show that at each step of the scheme,
\begin{equation}\label{eq1stationary}
u_N^{a,n}(t,x)\mathds 1_{t\I \tau_N^a\wedge \tau_N^b}=u_N^{b,n}(t,x)\mathds 1_{t\I \tau_N^a\wedge \tau_N^b}  \quad \text{a.s.} 
\end{equation}
For $n=0$, we clearly have $u_N^{a,0}(t,x)=u_N^{b,0}(t,x)=u_0(x)$. Now if \eqref{eq1stationary} holds for some $n\s 0$, then
\begin{align*}
 u_N^{a,n+1}(t,x)\mathds 1_{t\I \tau_N^a\wedge \tau_N^b}&=\mathds 1_{t\I \tau_N^a\wedge \tau_N^b}\int_0^t\int_{\R^d} g (t-s,x-y) \sigma\lp u_N^{a,n}(s,y) \rp\, L_N^a(\dd s , \dd y) \\
 &=\mathds 1_{t\I \tau_N^a\wedge \tau_N^b}\int_0^t\int_{\R^d} g (t-s,x-y) \sigma\lp u_N^{b,n}(s,y) \rp \mathds 1_{s\I \tau_N^a\wedge \tau_N^b} \,L_N^a(\dd s , \dd y)\\
 &=\mathds 1_{t\I \tau_N^a\wedge \tau_N^b}\int_0^t\int_{\R^d} g (t-s,x-y) \sigma\lp u_N^{b,n}(s,y) \rp \mathds 1_{s\I \tau_N^a\wedge \tau_N^b} \,L_N^b(\dd s , \dd y)\\
 &=\mathds 1_{t\I \tau_N^a\wedge \tau_N^b}\int_0^t\int_{\R^d} g (t-s,x-y) \sigma\lp u_N^{b,n}(s,y) \rp\, L_N^b(\dd s , \dd y) \\
 &=  u_N^{b,n+1}(t,x)\mathds 1_{t\I \tau_N^a\wedge \tau_N^b}\, .
\end{align*}
Since $u_N^{a,n}(t,x) \to u_N^{a}(t,x)$ and $u_N^{b,n}(t,x) \to u_N^{b}(t,x)$ as $n\to +\infty$ in $L^p(\Omega)$, we deduce \eqref{asequality}. 
 \end{proof}
 
 In \cite[Definition 5.1]{dalang99}, the second author introduced the property (S) for a stochastic process and a martingale measure, which is a sort of stationarity property in the space variable. In our case, the noise is not necessarily a martingale measure, but we can use a similar definition:
\begin{df}\label{propS}
 We say the family of random fields $u_N^a$ has \emph{property (S)} if the law of the process 
\begin{equation*}
 \lp \lp u_N^a(t,a+x) \, , (t,x)\in [0,T]\times \R^d \rp \, ; \lp L_N^a\lp [0,t]\times \lp a+B\rp \rp \, , (t,B) \in [0,T]\times \m B_b(\R^d) \rp \rp\, ,
\end{equation*}
does not depend on $a$.
\end{df}

\begin{lem}\label{eqinlaw}
 If $u_0(x)\equiv u_0$ is constant, then the family  $(u_N^a\colon a\in\R^d)$ has property (S).
\end{lem}

\begin{proof}
Similarly to the proof of Proposition~\ref{lem1}, it is enough to show property (S) for the Picard iterates $u_N^{a,n}$ for each $n\s 0$. 
For $n=0$, we obviously have $u^{a,0}_N(t,x+a)=u_0=u^{0,0}_N(t,x)$. So property (S) for $u_N^{a,n}$ follows from the fact that the law of $\lp L_N^a\lp [0,t]\times \lp a+B\rp \rp, (t,B) \in [0,T]\times \m B_b(\R^d) \rp$ does not depend on $a$. Next, assume that $u^{a,n}_N$ has the property (S). Since 
$$ u_N^{a,n+1}(t,x)= u_0 + \int_0^t\int_{\R^d} g (t-s,x-y) \sigma\lp u_N^{a,n}(s,y)\rp\, L_N^a(\dd s , \dd y)\,,$$
we can use the same argument as in \cite[Lemma 18]{dalang99}, since the proof only relies on the fact that $L$ has a law that is invariant under translation in the space variable. 
\end{proof}

\begin{theo}\label{spacestationary}
 If $u_0(x)\equiv u_0$ is constant, for any $a\in \R$, the random field $(u(t,a+x)\colon (t,x)\in [0,T]\times \R^d)$ has the same law as the random field $(u(t,x)\colon (t,x)\in [0,T]\times \R^d)$. 
\end{theo}
\begin{proof}
By \eqref{asequality}, $u_N^a(t,a+x)\mathds 1_{t\I \tau_N^a\wedge \tau_N^0}= u_N^0(t,a+x)\mathds 1_{t\I \tau_N^a\wedge \tau_N^0}$ almost surely.  Taking the stationary limit as $N\to +\infty$, we get that $u^a(t,a+x)=u^0(t,a+x)$ almost surely for any $(t,x)\in [0,T]\times \R^d$. Also, by the property (S) of the family of random fields $(u_N^a\colon a\in \R^d)$ (see Lemma \ref{eqinlaw}), the random field $(u_N^a(t,a+x)\colon (t,x)\in [0,T]\times \R^d)$ has the same law as the random field $(u_N^0(t,x)\colon (t,x)\in [0,T]\times \R^d)$. Again, taking the stationary limit as $N\to +\infty$, we get that the random field $(u^a(t,a+x)\colon (t,x)\in [0,T]\times \R^d)$ has the same law as the random field $(u^0(t,x)\colon (t,x)\in [0,T]\times \R^d)$. Therefore, the random field $(u^0(t,a+x)\colon (t,x)\in [0,T]\times \R^d)$ has the same law as the random field $(u^0(t,x)\colon (t,x)\in [0,T]\times \R^d)$.
\end{proof}
 

\subsubsection{Existence of a càdlàg solution in $H_{r,loc}(\R^d)$ with $r<-\frac d 2$}\label{frac_sobolev_space}
In the following, we want to establish a regularity result for the paths of the mild solution to \eqref{SHE} in the case $D=\R^d$, analogous to Theorem~\ref{cadlagsolution-general} which concerns $D=[0,\pi]$. Since $D=\R^d$ is unbounded, and the solution may not decay in space (see Theorem~\ref{spacestationary}), we consider the mild solution $u\colon t \mapsto u(t, \cdot)$ as a distribution-valued process in a local fractional Sobolev space, and prove that it has a càdlàg version in this space.

Recall to this end the \emph{Schwartz space} $\Scd$ of smooth functions $\varphi\colon \R^d\to\R$ such that 
$\sup_{x\in \R^d}\left | x^\alpha \varphi^{(\beta)}(x) \right | <+\infty$ for any multi-indices $\alpha, \, \beta \in \N^d$, equipped with the topology induced by the semi-norms 
$\sum_{|\alpha|, |\beta| \I p} \sup_{x\in \R^d} \left | x^\alpha \varphi^{(\beta)}(x) \right |$ for $p\in\N$. Here, $\varphi^{(\beta)}(x) =\partial_{x_1}^{\beta_1}\dots\partial_{x_d}^{\beta_d}\varphi(x)$ if $\beta=(\beta_1,\dots,\beta_d)$.
Its topological dual is called the space of \emph{tempered distributions} and is denoted by $\temd$. The classical \emph{Fourier transform} $\m F( \varphi)(\xi):= \int_{\R^d} e^{-i\xi\cdot x}\varphi(x) \,\dd x$ with $\xi \in \R^d$ and $\varphi\in\Scd$ 
can be extended by duality to $f\in\temd$:
	\begin{equation*}
		\scal{\m F(f)}{\varphi}:= \scal{f}{\m F (\varphi)}\, , \quad  \varphi \in \Scd\, .
	\end{equation*}

\begin{df}\label{hsspace}
	The \emph{(local) fractional Sobolev space of order $r\in \R$} is defined by
	\begin{align*}
		H_r(\R^d)&:=\left \{ f\in \m S'(\R^d) \colon \xi \mapsto \lp 1+|\xi | ^2\rp ^{\frac r 2} \m F (f) (\xi) \in L^2(\R^d) \right \}\\
		\bigg(~	H_{r,\text{loc}}(\R^d)&:=\left \{ f\in \m S'(\R^d) \colon \lp \forall \theta \in C^\infty_c(\R^d) \colon \theta f \in H_r(\R^d) \rp \right \} ~ \bigg)\,. 
	\end{align*}
The topology on $H_r(\R^d)$ is induced by the norm $$\left \| f\right \|_{H_r(\R^d)}:=\left \| (1+|\cdot |^2)^{\frac r 2} \m F(f)(\cdot)  \right \|_{L^2(\R^d)}\,,$$ and we have $f_n\to f$ in $H_{r,loc}(\R^d)$ if $\theta f_n\to \theta f$ in $H_{r}(\R^d)$ as $n\to +\infty$ for any $\theta \in C^\infty_c(\R^d)$. 
\end{df}

We now proceed to studying the regularity of $u\colon t \mapsto u(t, \cdot)$ in $H_{r,{loc}}(\R^d)$. As in the case of a bounded interval in dimension one, the drift part is easy to handle.
\begin{lem}\label{drift-bdd}
	\label{drift}
	Let $Z$ be a bounded measurable random field and
	$$F(t,x):= \int_0^t \int_{\R^d} g(t-s, x-y) Z(s,y)\, \dd s\, \dd y\, .$$
	Then the process $t\mapsto F(t,\cdot)$ is continuous in $\hs$ for any $r\I 0$.
\end{lem}

\begin{proof}
	Since $g\in L^1([0,T]\times\R^d)$ and $Z$ is bounded, it follows from \cite[Corollary~3.9.6]{Bogachev} that the sample paths of $F$ are jointly continuous in $(t,x)$ almost surely. Therefore, $t\mapsto F(t,\cdot)$ is continuous in $H_{0,{loc}}(\R^d)=L^2_{{loc}}(\R^d)$, hence also in $\hs$ for any $r\I 0$. 
\end{proof}

Next, we consider the situation where $\sigma$ is a bounded function. Already in this restricted case, the unboundedness of space and the possibility of having infinitely many large jumps require a more careful analysis of the different parts of the solution.

\begin{prop}\label{cadlagsolutionbddsigma}
Let  $\sigma$ be a bounded and Lipschitz function, and $u$ be the mild solution to the stochastic heat equation \eqref{SHE} constructed in Proposition \ref{existencerealline} under hypothesis \hyperlink{hyp1}{\textbf{(H)}}. Then, for any $r<-\frac d 2$, the stochastic process $(u(t, \cdot))_{t\in [0,T]}$ has a càdlàg version in $H_{r,\text{loc}}(\R^d)$. 
\end{prop}
\begin{proof} 
	Since the mild solution $u_N$ to the truncated equation \eqref{truncatedmild} agrees with the mild solution $u$ to the stochastic heat equation \eqref{SHE} on $\left \{t \I \tau_N\right \}$ (see the last statement of Proposition~\ref{existencerealline}), the sample path properties of $u$ and $u_N$ are the same, and we can restrict to the study of the regularity of the sample paths of $u_N$. Furthermore, there is no loss of generality if we take $N=1$. Therefore, we suppose that 
	$$u(t,x)=V(t,x)+\int_0^t \int_{\R^d} g(t-s,x-y) \sigma (u(s,y)) \, L_1(\dd s, \dd y)\, ,$$
	where $L_1$ is the truncated noise from \eqref{LN} with $N=1$. We use the decomposition
\begin{equation}\label{def-u}
u(t,x)=V(t,x)+u^{1,1}(t,x)+u^{1,2}(t,x) +u^{2,1}(t,x)+u^{2,2}(t,x) +u^3(t,x)\,,
\end{equation}
where for $A>0$ and $Z(s,y):=\sigma(u(s,y))$, 
\begin{align*}
u^{1,1}(t,x)&:=\int_0^t \int_{\R^d} g(t-s, x-y)Z(s,y) \mathds 1_{y \in [-2A, 2A]^d}\, L^M(\dd s , \dd y)\,,\\
u^{1,2}(t,x)&:=\int_0^t \int_{\R^d} g(t-s, x-y)Z(s,y) \mathds 1_{y \notin [-2A, 2A]^d} \, L^M(\dd s, \dd y)\,,\\
u^{2,1}(t,x)&:=\int_0^t \int_{\R^d}  g(t-s, x-y)Z(s,y)  \mathds 1_{y \in [-2A, 2A]^d}  \, L_1^P(\dd s , \dd y)\,,\\
u^{2,2}(t,x)&:=\int_0^t \int_{\R^d}  g(t-s, x-y)Z(s,y)  \mathds 1_{y \notin [-2A, 2A]^d}  \, L_1^P(\dd s , \dd y)\,,\\
u^{3}(t,x)&:=b\int_0^t \int_{\R^d}  g(t-s, x-y)Z(s,y)\,\dd s\,\dd y\, ,
\end{align*}
where $L^M$ is defined in \eqref{noise}, and $L^P_1$ is the noise obtained by applying the truncation \eqref{LN} with $N=1$ to $L^P$ from \eqref{noise}.
It is clear that $V$ is jointly continuous in $(t,x)$, and the same holds for $u^3$ as pointed out in the proof of Lemma~\ref{drift}. Furthermore, on $[0,T]\times [-2A,2A]^d$, the noise $L^P_1$ consists of only finitely many jumps. So upon a change of the drift $b$ and increasing the truncation level for $L^M$ from $1$ to the largest size of these jumps (which clearly does not affect the arguments below), we may assume that $u^{2,1}=0$. The remaining terms are now treated separately. 

\vspace{0.25\baselineskip}
\noindent
\underline{$u^{1,1}(t,x)$:} 
By definition of the Fourier transform, we have for 
$\varphi \in \Scd$,
\begin{equation*}
\begin{aligned}
\scal{\m F\lp u^{1,1}(t,\cdot)\rp}{\varphi} &= \scal{ u^{1,1}(t,\cdot)}{\m F \lp \varphi \rp}=\int_{\R^d} u^{1,1}(t,x)\m F(\varphi) (x) \,\dd x\\
 & =\int_{\R^d} \lp \int_0^t \int_{\R^d} g(t-s, x-y)Z(s,y)  \mathds 1_{y \in [-2A, 2A]^d} \,L^M(\dd s , \dd y)\rp \m F(\varphi) (x)\, \dd x\, .
\end{aligned}
\end{equation*}
Permuting the stochastic integral and the Lebesgue integral (because $\int_{|z|\I 1} |z|^p \,\nu(\dd z)<+\infty$, $g\in L^p([0,T]\times\R^d)$ and $Z$ is bounded, this  is possible by Theorem~\ref{fubini} together with the estimate \eqref{p-est}) yields
\begin{equation}\label{fourier-calc}
	\begin{aligned}
\scal{\m F( u^{1,1}(t,\cdot))}{\varphi} &=\int_0^t \int_{[-2A, 2A]^d}  \lp  \int_{\R^d}\m F(\varphi)(x)  g(t-s, x-y) \,\dd x \rp Z(s,y)  \,L^M(\dd s , \dd y)\\
& =\int_0^t \int_{[-2A, 2A]^d} \lp  \int_{\R^d} e^{-i\xi\cdot y -(t-s)|\xi |^2} \varphi(\xi) \, \dd \xi \rp Z(s,y)   \,L^M(\dd s , \dd y)\\
&= \int_{\R^d} \lp \int_0^t \int_{[-2A, 2A]^d}   e^{-i\xi\cdot y -(t-s)|\xi |^2}  Z(s,y)   \,L^M(\dd s , \dd y\rp \, \varphi(\xi) \, \dd \xi  \,, 
\end{aligned}
\end{equation}
which implies that $\m F( u^{1,1}(t, \cdot)) (\xi)$ is given by
\begin{equation}\label{fourier}
	\begin{aligned}
	a_\xi(t):=e^{-|\xi|^2 t}\int_0^t \int_{\R^d} e^{-i\xi \cdot y}e^{ s |\xi|^2} Z(s,y)  \mathds 1_{y \in [-2A, 2A]^d} \,L^M(\dd s , \dd y)\,.
	\end{aligned}
\end{equation}
 Thus,
\begin{equation*}
 \left \| u^{1,1}(t\pm h,\cdot)-u^{1,1}(t,\cdot)\right \|_{H_r(\R^d)}^2=\int_{\R^d} (1+|\xi|^2)^r \left | a_\xi(t\pm h)-a_\xi(t) \right |^2\,\dd \xi\, .
\end{equation*}
Since the function $t\mapsto e^{-|\xi|^2 t}$ is continuous, and the stochastic integral in $a_\xi(t)$ exists in $L^2(\Omega)$, $t\mapsto a_\xi(t)$ is continuous in $L^2(\Omega)$. Furthermore,
\begin{equation*}
 \E \lc\left | a_\xi(t) \right |^2 \rc \I C \frac{1-e^{-2|\xi|^2 t}}{2|\xi |^2} \I C\, ,
\end{equation*}
for some constant $C$ that does not depend on $\xi$, so by the dominated convergence theorem (which applies since $r<-\frac d 2$),
\begin{equation*}
 \E \lc \left \| u^{1,1}(t+h,\cdot)-u^{1,1}(t,\cdot)\right \|_{H_r(\R^d)}^2 \rc \to 0 \, , \qquad \text{as} \ h\to 0\, ,
\end{equation*}
and the process $t\mapsto u^{1,1}(t,\cdot)$ is continuous in $L^2(\Omega)$ as a process with values in $H_r(\R^d)$.

In order to apply \cite[Chapter III, \S4, Theorem~1]{skorohod} to deduce the existence of a càdlàg modification of $t\mapsto u^{1,1}(t,\cdot)$ in $H_r(\R^d)$ for any $r<-\frac d 2$, it remains to prove 
\begin{equation*}
	\E\lc \|u^{1,1}(t+h, \cdot)-u^{1,1}(t,\cdot)\|_{H_r(\R^d)}^2 \| u^{1,1}(t-h, \cdot)-u^{1,1}(t,\cdot)\|_{H_r(\R^d)}^2\rc \I C h^{1+\delta}
\end{equation*}
for some $\delta>0$. Upon defining, similar to \eqref{int}, 
\begin{equation*}
	I_a^b(\xi):=\int_a^b \int_{\R^d} \int_\R e^{-i\xi \cdot y}e^{ s |\xi|^2} Z(s,y) z\mathds 1_{|z|\I 1} \mathds 1_{y \in [-2A, 2A]^d} \,\tilde J(\dd s , \dd y, \dd z)
\end{equation*}
for $0\I a<b \I T$ and $\xi \in \R^d$, the proof is identical to that of Proposition \ref{cadlagsolutioninterval} for the equation on a bounded interval if we make the following replacements:
\begin{align*}
 [0,\pi] \longleftrightarrow [-2A, 2A]^d \,,\quad k \longleftrightarrow \xi \, ,\quad
 \sin(ky) \longleftrightarrow e^{-i \xi\cdot y}\, .
\end{align*}

\vspace{0.25\baselineskip}
\noindent
\underline{$u^{1,2}(t,x)$:} If $f\colon\R^d\to \R$ is a smooth function, then for  $a,b \in \R^d$ with $a_i\I b_i$ for all $1\I i \I d$,
\begin{equation}\label{formula}
  f(b)= f(a)+\sum_{i=1}^d \, \sum_{1\I k_1<\dots <k_i\I d}\, \int_{a_{k_1}}^{b_{k_1}} \dd r_{k_1} \dots \int_{a_{k_i}}^{b_{k_i}} \dd r_{k_i} \partial_{x_{k_1} \dots x_{k_i}} f( c_{\mathbf k}(a,r) ) \, ,
\end{equation}
where for $\mathbf k=(k_1, \dots, k_i)$ with $k_1<\dots < k_i$, we define $\lp c_{\mathbf k}(a,r)\rp_j:= a_i \mathds 1_{j\notin \mathbf k} +r_j \mathds 1_{j \in \mathbf k}$ for $1\I j\I d$. This formula is easily proved by induction on the dimension. Since the heat kernel $g(t-s,x-y)$ is smooth on $y\notin [-2A, 2A]^d$ for $x\in [-A, A]^d$, \eqref{formula} with $a=(s,-A, \dots , -A)$ and $b=(t,x)$ gives
\begin{equation*}
\begin{aligned}
  g(t-s,x-y) &=\sum_{i=1}^{d+1} \, \sum_{1\I k_1<\dots <k_i\I d+1} \,  \int_{a_{k_1}}^{b_{k_1}} \dd r_{k_1} \dots \int_{a_{k_i}}^{b_{k_i}} \dd r_{k_i}  \partial_{x_{k_1} \dots x_{k_i}} g ( c_{\mathbf k}(a,r)-(s,y) )\\
 & =\sum_{i=1}^{d} \, \sum_{1\I k_1<\dots <k_i\I d} \, \int_s^t \dd u \int_{-A}^{x_{k_1}} \dd r_{k_1} \dots \int_{-A}^{x_{k_i}} \dd r_{k_i}  \partial_{x_{k_1} \dots x_{k_i}} \partial_t g( u-s , c_{\mathbf k}(-\mathbf A,r)- y ) \, ,
\end{aligned}
\end{equation*}
where $\mathbf A:= (A,\dots , A)$.
Another application of Theorem~\ref{fubini} and \eqref{p-est} shows that $u^{1,2}(t,x)$ equals
\begin{equation}\label{Fubini-used}
\begin{aligned}
&\sum_{i=1}^{d} \, \sum_{1\I k_1<\dots <k_i\I d} \, \int_0^t \dd u \int_{-A}^{x_{k_1}} \dd r_{k_1} \dots \int_{-A}^{x_{k_i}} \dd r_{k_i} \\
 &\qquad \lp \int_0^u \int_{\R^d}   \partial_{x_{k_1}}\dots \partial_{x_{k_i}} \partial_t g\lp u-s , c_{\mathbf k}(-\mathbf A,r)- y \rp Z(s,y) \mathds 1_{y \notin [-2A, 2A]^d} \, L^M(\dd s , \dd y) \rp\, .
\end{aligned}
\end{equation}
We see from this expression that $u^{1,2}$ is jointly continuous in $(t,x)$. By the argument at the end of the proof of Lemma \ref{drift}, we deduce that $t\mapsto u^{1,2}(t, \cdot)\mathds 1_{[-A,A]^d}$ is continuous in $H_r(\R^d)$ for  $r\I 0$.

\vspace{0.25\baselineskip}

\noindent
\underline{$u^{2,2}(t,x)$:} This process takes into account only the jumps that are far away from $x$, but that can be arbitrarily large. We can write $u^{2,2}$ as a sum:
\begin{align*}
 u^{2,2}(t,x)&=\sum_{i\s 1} g(t-T_i, x-X_i) Z(T_i, X_i) Z_i \mathds 1_{X_i \notin [-2A, 2A]^d \, , \, 1< |Z_i| < 1+ |X_i|^\eta \, , \, T_i \I t}\, .
\end{align*}
We first observe that each term of this sum is jointly continuous in $(t,x) \in [0,T]\times [-A,A]^d$ almost surely. We show that this sum converges uniformly in $(t,x)\in [0,T]\times [-A,A]^d$. Choose $A$ large enough such that $T< \frac{A^2}{2d} $. Because $|x-X_i|>A$, Lemma~\ref{maxheatkernel} below shows that the maximum of the function $t\mapsto g(t, x-X_i)$ is attained at $t=T$:
\begin{align*}
 \sup_{t\I T, x\in [-A,A]^d} g(t-T_i, x-X_i) &\I \sup_{x\in [-A,A]^d}  C {T^{-\frac d 2}} e^{-\frac{|x-X_i|^2}{4T}}\I  C {T^{-\frac d 2}} e^{-\frac{|p_A(X_i)-X_i|^2}{4T}}\, ,
\end{align*}
where $p_A$ is the projection on the convex set $[-A, A]^d$. Then, for $\beta=1 \wedge q$,
\begin{equation}\label{u22}
\begin{aligned}
 \E & \Bigg[  \Bigg( \sum_{i\s 1} \sup_{t\I T, x\in [-A,A]^d} \left | g(t-T_i, x-X_i) Z(T_i, X_i) Z_i \mathds 1_{X_i \notin [-2A, 2A]^d \, , \, 1< |Z_i| < 1+ |X_i|^\eta \, , \, T_i \I t} \right | \Bigg)^\beta \Bigg]\\ 
 &\hspace{1cm} \I   \frac C {T^{\frac {\beta d} 2}} \E \Bigg[ \Bigg(  \sum_{i\s 1}  \left | e^{-\frac{|p_A(X_i)-X_i|^2}{4T}} Z_i \mathds 1_{X_i \notin [-2A, 2A]^d \, , \, 1< |Z_i| \, , \, T_i \I T} \right | \Bigg)^\beta \Bigg] \\
 &\hspace{1cm} \I  \frac C {T^{\frac {\beta d} 2}} \E \lc  \sum_{i\s 1}  \left | e^{-\frac{|p_A(X_i)-X_i|^2}{4T}} Z_i \mathds 1_{X_i \notin [-2A, 2A]^d \, , \, 1< |Z_i| \, , \, T_i \I T } \right |^\beta\rc \\
 & \hspace{1cm} \I C \int_0^T \int_{y\notin [-2A, 2A]^d} \int_{|z|>1} |z|^\beta e^{-\beta\frac{|p_A(y)-y|^2}{4T}}\, \dd s \, \dd y \, \nu( \dd z) <+\infty\, .
\end{aligned}
\end{equation}
Therefore, the sum defining $u^{2,2}$ converges uniformly in $(t,x) \in [0,T]\times [-A,A]^d$, and $u^{2,2}$ is jointly continuous. Thus, $t\mapsto u^{2,2}(t, \cdot)\mathds 1_{[-A,A]^d}$ is continuous in $H_r(\R^d)$ for every $r\I 0$.

Since $A$ can be chosen arbitrarily large, the assertion of the proposition follows.
\end{proof}

In order to pass from bounded to unbounded nonlinearities $\sigma$, the basic strategy remains the same as in the proof of Theorem~\ref{cadlagsolution-general}. However, it was crucial in that proof that the solution have a finite second moment. Unfortunately, in dimensions $d\s 2$, the mild solution $u$ to \eqref{SHE} has no finite second moments as a result of the singularity of $g$. And it is easy to convince oneself that taking powers $p<2$ instead of $2$ does not combine well with the $\|\cdot\|_{H_r(\R^d)}$-norms. Instead, in the proof we propose below, the idea is to consider an equivalent probability measure $\Q$ (which obviously does not affect the path properties of $u$) under which the solution has a finite second moment. Although $L$ might not be a Lévy noise under $\Q$ anymore, it follows from the theory of integration against random measures, which we briefly recall in the Appendix, that there exists a particularly clever choice of $\Q$ such that we have sufficient control on the second moments of both integrands and integrators under $\Q$. 

\begin{theo}\label{cadlagsolution}
If $u$ is the mild solution to the stochastic heat equation \eqref{SHE} constructed under the assumptions of Proposition \ref{existencerealline}, then, for any $r<-\frac d 2$, the stochastic process $(u(t, \cdot))_{t\in [0,T]}$ has a càdlàg version in $H_{r,\text{loc}}(\R^d)$. 
\end{theo}

\begin{proof} We first consider the case $p\s 1$ in assumption \hyperlink{hyp1}{\textbf{(H)}}.
As in Proposition~\ref{cadlagsolutionbddsigma}, we can suppose that $u$ is the solution to \eqref{truncatedmild} with $N=1$, and use the decomposition 
\eqref{def-u} with $A>0$. The terms $V$ and $u^{2,1}$ can be dealt with as in Proposition~\ref{cadlagsolutionbddsigma}. For the remaining terms, we use different arguments.


\vspace{0.25\baselineskip}

\noindent
\underline{$u^{1,1}(t,x)$:} 
 Let $\sigma_n(u)= \sigma(u)\mathds 1_{|u|\I n}$ and define $u^{1,1}_n$ as in \eqref{def-u} but with $Z$ replaced by $\sigma_n(u)$.
Then, 
\begin{equation*}
 u^{1,1}(t,x)-u^{1,1}_n(t,x)= \int_0^t \int_{y\in [-2A,2A]^d} g(t-s, x-y)\lp \sigma(u(s,y))- \sigma_n(u(s,y)) \rp\, L^M(\dd s , \dd y)\, ,
\end{equation*}
and
\begin{align}\label{hsnorm2}
 \| u^{1,1}(t,\cdot)-u^{1,1}_n(t,\cdot)\|_{H_r(\R^d)}^2=\int_{\R^d} (1+|\xi |^2)^r  \left | \m F\lp u^{1,1}(t,\cdot)-u^{1,1}_n(t,\cdot) \rp (\xi )\right |^2 \,\dd \xi\, .
\end{align}
Writing $\sigma_{(n)} (s,y)= \sigma(u(s,y))- \sigma_n(u(s,y))$, we obtain
\begin{equation*}
\m F(u^{1,1}(t,\cdot)-u^{1,1}_n(t,\cdot))(\xi )= \int_0^t \int_{[-2A,2A]^d}  e^{-i\xi  \cdot y}e^{ -(t-s) |\xi|^2}\sigma_{(n)}(s,y)  \,L^M(\dd s , \dd y)
\end{equation*}
as in \eqref{fourier}. With similar calculations as in \eqref{calc-akn}, but using Theorem~\ref{fubini} with $1<p<1+\frac 2 d$, one can show that
\begin{equation}\label{Fu11}
\begin{aligned}
\sup_{t\in[0,T]} \left | \m F(u^{1,1}(t,\cdot)-u^{1,1}_n(t,\cdot) ) (\xi ) \right | 
&\I C \sup_{t \in [0,T]} \left | \int_0^t \int_{[-2A,2A]^d}e^{-i\xi  \cdot y} \sigma_{(n)} (s,y)\,L^M(\dd s, \dd y) \right |  \, ,
\end{aligned}
\end{equation}
where $C$ does not depend on $\xi$. With notation from the Appendix, the fact that $\sigma(u) \in L^{1,p}(L^M,\bbp)$ implies that there exists a probability measure $\Q$ that is equivalent to $\bbp$ such that the process $\sigma(u)$ belongs to $L^{1,2}(L^M,\Q)$, see Theorem~\ref{change}. 
Consequently, using the notation in \eqref{Daniell}, we deduce from \eqref{hsnorm2} that
\begin{equation}\label{EQ}
\begin{aligned}
	&\E_\Q\left[ \sup_{t\in[0,T]} \| u^{1,1}(t,\cdot)-u^{1,1}_n(t,\cdot)\|_{H_r(\R^d)}^2 \right] \\&\qquad\I C \int_{\R^d} (1+|\xi|^2)^r	\E_\Q\left[ \sup_{t \in [0,T]} \left | \int_0^t \int_{y\in [-2A,2A]^d}e^{-i\xi  \cdot y} \sigma_{(n)} (s,y)\,L^M(\dd s, \dd y) \right |^2\right] \,\dd \xi \\
	&\qquad\I C  \int_{\R^d} (1+|\xi|^2)^r \|e^{-i\xi\cdot (\cdot)}\sigma_{(n)}\|^2_{L^M,2,\Q} \,\dd \xi
	\I C\|\sigma_{(n)} \|^2_{L^M,2,\Q}\int_{\R^d} (1+|\xi|^2)^r\,\dd \xi\,.
\end{aligned} \end{equation}
The last integral is finite because $r<-\frac d 2$. Moreover, $\sigma_{(n)}(\omega,s,y) \to 0$, pointwise in $(\omega,s,y)$, and is bounded by $\sigma(u(\omega,s,y))$, which belongs to $L^{1,2}(L^M,\Q)$ by assumption. Hence, by Theorem~\ref{domcon}, the left-hand side of \eqref{EQ} converges to $0$ as $n\to+\infty$.
As before, we may extract a subsequence that converges uniformly in $[0,T]$ almost surely with respect to $\Q$, and hence $\bbp$. We now deduce that $t\mapsto u^{1,1}(t,\cdot)$ has a càdlàg modification because the processes $u_n^{1,1}$ have càdlàg modifications by Proposition~\ref{cadlagsolutionbddsigma}.

\vspace{0.25\baselineskip}

\noindent
\underline{$u^{1,2}(t,x)$:} The proof is identical to the corresponding part in Proposition~\ref{cadlagsolutionbddsigma}, provided we can still apply Theorem~\ref{fubini} in \eqref{Fubini-used}. In order to justify this, observe that
\begin{align*}
	&\E\left[ \int_0^u \int_{\R^d} \int_{|z|\I 1} |\partial_{x_{k_1}\dots x_{k_i}}\partial_t g\lp u-s , c_{\mathbf k}(-\mathbf A,r)- y \rp\sigma(u(s,y))z|^p  \mathds 1_{y \notin [-2A, 2A]^d} \,\dd s\,\dd y\,\nu(\dd z) \right]\\
	&\qquad\I C\int_0^u \int_{\R^d}  |\partial_{x_{k_1}\dots x_{k_i}}\partial_t g( u-s , c_{\mathbf k}(-\mathbf A,r)- y )|^p  E_{\frac{2-d (p-1)}{2(p\vee 1)},\frac{1}{p\vee 1}}\Big( C |y|^{\frac{\eta(p-q)}{p\vee1}}\Big)\mathds 1_{y \notin [-2A, 2A]^d} \,\dd s\,\dd y\\
	&\qquad \I C\int_0^u \int_{\R^d}  |\partial_{x_{k_1}\dots x_{k_i}}\partial_t g( u-s , c_{\mathbf k}(-\mathbf A,r)- y )|^p  \exp\Big(C|y|^{\frac{2\eta(p-q)}{2-d (p-1)}}\Big)P(|y|)\mathds 1_{y \notin [-2A, 2A]^d} \,\dd s\,\dd y
\end{align*}
by \eqref{moment1} and \cite[Theorems~4.3 and 4.4]{Gorenflo14} with some polynomial $P$. Next, for every multi-index $\alpha\in\N^{1+d}$, it is easily verified by induction that $\partial^\alpha g(t,x)$ takes the form $Q(t^{-\frac 1 2},x)\exp(-\frac{|x|^2}{4t})$ for some polynomial $Q$. So if $l$ denotes the degree of $Q$, we have for every $t\in[0,T]$ and $|x|^2\s 2lT$,
\[ |\partial^\alpha g(t,x)| \I C(1+t^{-\frac l 2}+|x|^l)e^{-\frac{|x|^2}{4t}} \I C(1+T^{-\frac l 2}+|x|^l)e^{-\frac{|x|^2}{4T}} =: \tilde Q(x)e^{-\frac{|x|^2}{4T}}\,.   \]

 Hence, as $|c_{\mathbf k}(-\mathbf A,r)- y|\s A$ for $y \notin [-2A, 2A]^d$, we obtain for sufficiently large $A$ that the expectation in the penultimate display is bounded by
 \begin{align*}
 	&C\int_{\R^d} |\tilde Q(c_{\mathbf k}(-\mathbf A,r)- y)|^p e^{-\frac{p|c_{\mathbf k}(-\mathbf A,r)- y|^2}{4T}}\exp\Big(C|y|^{\frac{2\eta(p-q)}{2-d (p-1)}}\Big)P(|y|)\mathds 1_{y \notin [-2A, 2A]^d} \,\dd y\\
 	&\qquad \I C\int_{\R^d} |\tilde Q(c_{\mathbf k}(-\mathbf A,r)- y)|^p e^{-\frac{p|c_{\mathbf k}(-\mathbf A,r)- y|^2}{4T}}\exp\Big(C|y|^{\frac{2\eta(p-q)}{2-d (p-1)}}\Big)P(|y|)\,\dd y\\
 	&\qquad =C\int_{\R^d} |\tilde Q(y)|^p e^{-\frac{p|y|^2}{4T}}\exp\Big(C|c_{\mathbf k}(-\mathbf A,r)- y|^{\frac{2\eta(p-q)}{2-d (p-1)}}\Big)P(|c_{\mathbf k}(-\mathbf A,r)- y|)\,\dd y\,.
 \end{align*}
Since $|c_{\mathbf k}(-\mathbf A,r)|\I \sqrt{d}A$, $|x-y|^b\I 2^{b-1} (|x|^b+|y|^b)$ for $b\s 1$, and $P(|x-y|)\I C(1+|x|^m + |y|^m)$ where $m$ is the degree of $P$, one can find another polynomial $\tilde P$ such that the last integral is further bounded by
\[ C\int_{\R^d} |\tilde Q(y)|^p e^{-\frac{p|y|^2}{4T}}\exp\Big(C|y|^{\frac{2\eta(p-q)}{2-d (p-1)}}\Big)\tilde P(y)\,\dd y\, , \]
which is independent of $r$, and finite because it is possible by assumption \hyperlink{hyp1}{\textbf{(H)}} to choose $\eta>\frac d q$ such that $\frac{2\eta(p-q)}{2-d (p-1)}< 2$ is satisfied. Theorem~\ref{fubini} is therefore applicable by \eqref{p-est}.

\vspace{0.25\baselineskip}

\noindent
\underline{$u^{2,2}(t,x)$:} The argument remains the same as in Proposition~\ref{cadlagsolutionbddsigma}, except that we have to replace the final bound in \eqref{u22} by
\[ C\int_0^T \int_{y\notin [-2A, 2A]^d} \int_{|z|>1} |z|^\beta e^{-\beta\frac{|p_A(y)-y|^2}{4T}}\left(E_{\frac{2-d (p-1)}{2(p\vee 1)},\frac{1}{p\vee 1}}\Big( C |y|^{\frac{\eta(p-q)}{p\vee1}}\Big)\right)^{\frac \beta p}\, \dd s \, \dd y \, \nu( \dd z)\, ,  \]
which is finite by an argument similar to the one for $u^{1,2}$.

\vspace{0.25\baselineskip}

\noindent
\underline{$u^3(t,x)$:} Consider the decomposition $u^3(t,x)=u^{3,1}(t,x)+u^{3,2}(t,x)$ where
\begin{equation}\label{u3-split}\begin{aligned} u^{3,1}(t,x)&=  b\int_0^t \int_{y\in [-2A, 2A]^d} g(t-s, x-y) \sigma(u(s,y))\, \dd s\, \dd y\, ,\\
u^{3,2}(t,x)&= b\int_0^t \int_{y\notin [-2A, 2A]^d} g(t-s, x-y) \sigma(u(s,y)) \, \dd s \,\dd y\,. \end{aligned}\end{equation}
If $u^{3,1}_n$ is the process obtained from $u^{3,1}$ by replacing $\sigma(u(s,y))$ by $\sigma_n(u(s,y))$, then, as in \eqref{Fu11}, 
\begin{align*}
	\sup_{t\in[0,T]} \left | \m F( u^{3}(t,\cdot)-u^{3}_n(t,\cdot) ) (\xi ) \right | 
	&\I C \sup_{t \in [0,T]} \left | \int_0^t \int_{\R^d}e^{-i\xi  \cdot y} \sigma_{(n)} (s,y) \mathds 1_{y\in [-2A, 2A]^d}\,\dd s\,\dd y \right| \\
	&\I C\int_0^T \int_{\R^d} |\sigma_{(n)}(s,y)|\mathds 1_{y\in [-2A, 2A]^d}\,\dd s\,\dd y\,.
\end{align*}
Consequently, we have
\begin{align*} \sup_{t\in[0,T]} \| u^{3,1}(t,\cdot)-u^{3,1}_n(t,\cdot)\|_{H_r(\R^d)}^2&=\sup_{t\in [0,T]}\int_{\R^d} (1+|\xi |^2)^r  \left | \m F( u^{3,1}(t,\cdot)-u^{3,1}_n(t,\cdot) ) (\xi )\right |^2 \,\dd \xi\\ & \I C \int_0^T \int_{\R^d} |\sigma_{(n)}(s,y)|\mathds 1_{y\in [-2A, 2A]^d}\,\dd s\,\dd y \int_{\R^d} (1+|\xi |^2)^r \,\dd \xi\, . \end{align*}
Recalling that $u$ is the solution to \eqref{truncatedmild} with $N=1$, the expectation of the left-hand side tends to $0$ as $n\to+\infty$ by \eqref{moment1} and the dominated convergence theorem. Hence, $t\mapsto u^{3,1}(t,\cdot)$ inherits the càdlàg sample paths of $u^{3,1}_n$, see Lemma~\ref{drift-bdd}.
Concerning $u^{3,2}$, the  continuity of $(t,x)\mapsto u^{3,2}(t,x)$ on $[0,T]\times[-A,A]^d$ is shown in the same way as for $u^{1,2}$. Instead of the stochastic Fubini theorem, one can use the ordinary Fubini theorem because
\begin{align*} &\int_0^t \dd u \int_{-A}^{x_{k_1}} \dd r_{k_1} \dots \int_{-A}^{x_{k_i}} \dd r_{k_i}\\
&\qquad \left(\int_0^u \int_{\R^d} |\partial_{x_{k_1}\dots x_{k_i}}\partial_t g\lp u-s , c_{\mathbf k}(-\mathbf A,r)- y \rp\sigma(u(s,y))|  \mathds 1_{y \notin [-2A, 2A]^d} \,\dd s\,\dd y\right)<+\infty \end{align*}
almost surely. This is verified by showing that the expectation of the integral in brackets is finite and uniformly bounded in $u$ and $r_{k_1}, \dots, r_{k_i}$. This concludes the proof for $p\s 1$.

For $0<p<1$, we have to modify the proof in the following way. Because $L$ has drift $b_0=0$ and summable jumps by the assumption $\int_{|z|\I 1} |z|^p \,\nu(\dd z)<+\infty$, we can write $u$ in the same form as \eqref{def-u} with $L^M(\dd t,\dd x)$ replaced by $\int_{|z|\I 1} z \, J(\dd t,\dd x,\dd z)$ and $u^3 = 0$. An inspection of the proof above shows that the arguments for $V$, $u^{2,1}$, $u^{2,2}$ remain valid, and in principle also for $u^{1,1}$ and $u^{1,2}$ if changing the order of integration in \eqref{fourier-calc} and \eqref{Fubini-used}, respectively, is permitted. The justification is comparable to the situation for $p\s 1$; one only has to use \eqref{p-est-2} instead of \eqref{p-est}: 
\begin{align*}
	&\int_0^t\int_{\R^d}\int_{{ |z|\I 1}}  \lp  \int_{\R^d}|\m F(\varphi)(x) | g(t-s, x-y) \,\dd x \rp^p  \E[\sigma(u(s,y))|^p] |z|^p \mathds 1_{y \in [-2A, 2A]^d} \,\dd s\,\dd y\,\nu(\dd z)\\
	&\qquad \I C\int_0^T\int_{\R^d} \lp \int_{\R^d} g(t-s, x-y) \,\dd x \rp^p\mathds 1_{y \in [-2A, 2A]^d} \,\dd s\,\dd y = CT(4A)^d <+\infty\, .
\end{align*}
\end{proof}

\begin{rem}\label{remark}
	The paper \cite{Hausenblas05} studies the existence of càdlàg modifications in certain Banach spaces of solutions to a class of stochastic PDEs driven by Poisson random measures. Example~2.3 in \cite{Hausenblas05} particularizes to the case of the stochastic heat equation with a multiplicative Lévy space--time white noise. However, this example contains an error since the measure $\nu$ in the first display on p.\ 1502 is not a Lévy measure (it is infinite on sets of the form $\{|x|>\delta\}$ for all sufficiently small values of $\delta$, contradicting Remark~3.1 in \cite{Hausenblas05}). After private communication with the author, it seems that this  example could be rewritten for the case of a bounded domain, but cannot be extended to the case where $D=\R^d$ (because the stopping times in \eqref{tauN} with $[0,\pi]$ replaced by $\R^d$ are $0$ almost surely for all $N\in\N$, cf.\ the discussion at the beginning of Section~\ref{equation_whole_space}).
\end{rem}


 
 \subsection{The stochastic heat equation on bounded domains}\label{equation_bdd_domain}
Let $D$ be a $C^\infty$-regular domain of $\R^d$, where $d\s 2$, that is, we assume that $D$ is a bounded open set whose boundary $\partial D$ is a smooth $(d-1)$-dimensional manifold, and whose closure $\bar D$ has the same boundary $\partial \bar D=\partial D$. For the stochastic heat equation \eqref{SHE} on such a domain $D$, 
we assume:
\begin{itemize}
\item[ \hypertarget{hyp2}{\textbf{(H')}}] There exists $0<p<1+\frac 2 d$ such that $\int_{|z|\I 1}|z| ^p \,\nu(\dd z)<+\infty$.
\end{itemize}
As in the case of an interval (Section~\ref{bddinterval}), the stopping times 
\begin{equation}\label{tauND} \tau_N=\inf\left \{ t\in [0,T]\colon J\lp [0,t]\times D \times [-N,N]^c \rp\neq 0\right \}\end{equation} are almost surely strictly positive and equal to $+\infty$ for large $N$. 

\begin{prop}\label{existencebdddomain}
 Let $D$ be a $C^\infty$-regular domain, $\sigma\colon\R\to \R$ be a Lipschitz function and let $L$ be a pure jump  
 Lévy white noise as in \eqref{noise} such that \hyperlink{hyp2}{\textbf{(H')}} is satisfied. Then there exists a predictable mild solution $u$ to \eqref{SHE} such that for all $0<p<1+\frac{2}{d}$,
\begin{equation}\label{moment-D}
  \sup_{(t,x)\in [0,T]\times D}\E\lc |u(t,x)|^{p} \mathds 1_{t\I \tau_N}\rc <+\infty\,.
\end{equation}
Furthermore, up to modifications, the solution is unique among all predictable random fields that satisfy \eqref{moment-D}.
\end{prop}
\begin{proof} By \cite[Corollary 3.2.8]{davies}, 
	$G_D(t;x,y) \I C{t^{-\frac d 2}} e^{-\frac{|x-y|^2}{6t}}$,
	so \cite[Theorem 3.5]{Chong2016} applies.
\end{proof}

As in the proof of Proposition~\ref{cadlagsolutioninterval}, the stopping times $\tau_N$ allow us to ignore the big jumps
for the analysis of path properties of the solution. So we only need to consider 
\begin{equation}\label{truncatednoisebdd}
	L_N( \dd t, \dd x)= b_N\,\dd t\,\dd x+\int_{|z|\I N} z \,\tilde J(\dd t, \dd x, \dd z)\, ,
\end{equation}
where $b_N := b-\int_{1<|z| \I N}z \,\nu(\dd z)$, and the corresponding mild solutions to
\begin{equation}\label{uN-D}
	u_N(t,x)=V(t,x)+\int_0^t \int_D G_D(t-s; x,y) \sigma (u_N(s,y))\, L_N(\dd s , \dd y) \,.
	\end{equation}
For simplicity, we take $N=1$ in the following, so that our equation becomes
\begin{equation}\label{newsolutiondomain}
\begin{aligned}
  u(t,x)&=V(t,x)+ b\int_0^t \int_D G_D(t-s;x,y) \sigma(u(s,y)) \, \dd s \, \dd y\\
  &\quad\qquad+ \int_0^t \int_D G_D(t-s;x,y) \sigma(u(s,y)) L^M(\dd s , \dd y)\, .
\end{aligned}
\end{equation}


\subsubsection{The fractional Sobolev spaces $H_r(D)$}\label{fracsobolevdef}
The operator $-\Delta$ on $D$ with vanishing Dirichlet boundary conditions admits a complete orthonormal system in $L^2(D)$ of smooth eigenfunctions $(\Phi_j)_{j\s 1}$, with eigenvalues $(\lambda_j)_{j\s 1}$. Then we have the following properties (see for example \cite[Chapter V, p. 343]{spdewalsh}):
\begin{align}\label{summable}
\sum_{j\s 1} (1+\lambda_j)^r <+\infty \, , \quad \text{for any} \ r<-\frac d 2\, , \\
\label{boundedestimate}
 \left \| \Phi_j \right \|_{L^\infty(D)}\I C (1+\lambda_j)^{\frac \alpha 2} \, , \quad \text{for any} \ \alpha>\frac d 2\, .
\end{align}
The Green's function $G_D$ has the representation \eqref{GD-spectral}
and we have the decomposition
\begin{equation}\label{l2decomposition}
f(x)=\sum_{j\s 1}a_j(f) \Phi_j(x)\, , \quad   x\in D\, ,
\end{equation}
for every $f\in L^2(D)$ 
where $a_j(f)=\scal{f}{\Phi_j}_{L^2(D)}$. 
For  $r\s 0$, we now define 
\begin{equation*}
 H_r(D):=\Bigg\{ f\in L^2\lp D \rp\colon  \| f \|^2_{H_r} := \sum_{j\s 1}\lp 1+\lambda_j\rp ^r a_j(f)^2 <+\infty \Bigg\} \, ,
\end{equation*}
which becomes a Hilbert space with the inner product 
$\scal f h _{H_r}:= \sum_{j\s 1} \lp 1+\lambda_j\rp ^r a_j(f) a_j(h)$.
We denote by $H_{-r}\lp D \rp$ the topological dual space of $H_{r}\lp D \rp$, which turns out to be isomorphic to the space of sequences $b=(b_n)_{n\s 1}$ such that
\begin{equation*}
\| b \|^2_{H_{-r}}:= \sum_{j\s 1}\lp 1+\lambda_j \rp ^{-r} b_j^2 <+\infty\, .
\end{equation*}
In fact, with $b_j= \tilde f(\Phi_j)$ for $\tilde f\in H_{-r}\lp D \rp$, we have $\|\tilde f\|_{H_{-r}} = \| b \|_{H_{-r}}$ and the pairing between $H_{-r}\lp D \rp$ and $H_{r}\lp D \rp$ is given by
$$\scal{b}{h}=\sum_{j\s 1} b_j a_j(h)\I \|b\|_{H_{-r}} \, \|h\|_{H_r}\, .$$
We need the following technical lemma, for which we could not find a reference in the literature.
\begin{lem}\label{equivhsspace}
	For $r\I 0$, the restriction of  $H_r(\R^d)$ to $D$ is continuously embedded in $H_r(D)$.
\end{lem}
\begin{proof}
	For $m\in \N$, let
	$H^m(D)=\{ u \colon \ u^{(\alpha)} \in L^2(D) \ \text{for all} \ |\alpha |\I m \}$,
	with an upper index $m$, be the ``usual'' Sobolev spaces as in \cite[p.~3]{lions}.  For real $r\s 0$, let $m$ be the smallest even integer with $m\s r$. Following \cite[Chapitre~1, (9.1)]{lions}, we define, with a superscript index, 
		$$H^r(D):=\left [ H^m(D) , L^2(D)\right ]_{1-\frac r m}\, ,$$
	where the right-hand side is the notation of \cite[Chapitre~1, Définition~2.1]{lions} for interpolation spaces.
	Furthermore, define $H_0^r(D)$ for $r\s 0$ as the closure in $H^r(D)$ of the set of smooth functions with compact support in $D$, see \cite[Chapitre~1, (11.1)]{lions}. Similarly, as in \cite[Definition 8.1]{grisvard}, let $H^r_B(D)$ for $r>\frac 1 2$ be the closed subspace of $H^r(D)$ such that its elements are equal to zero on $\partial D$. 
	
	By the definition of interpolation spaces, there exists for each $\theta \in [0,1]$ some self-adjoint positive operator $\Lambda$ in $L^2(D)$ with domain $H^m_B(D)$ such that 
	$$\left [H^m_B(D), L^2(D) \right ]_\theta=\text{dom}\lp \Lambda^{1-\theta} \rp\, .$$
	The notion of domain is as in \cite[Chapitre~1, p.\ 12]{lions}, and the power in this case is to be understood as the spectral power of the operator $\Lambda$. By \cite[Chapitre~1, Remarque 2.3]{lions}, $\text{dom} (\Lambda^{1-\theta})$ coincides with $\text{dom}( \tilde \Lambda^{1-\theta})$ for any other self-adjoint positive operator $\tilde\Lambda$ in $L^2(D)$ with domain $H^m_B(D)$. In particular, we can choose $\tilde\Lambda=(-\Delta)^{\frac m 2}$, where $\Delta$ is the Dirichlet Laplacian, and the power $\frac m 2$, an integer because $m$ is even, is to be understood as the composition of partial differential operators. Then, from \cite[Théorème~8.1]{grisvard}, we deduce that
	$$\left [H^m_B(D), L^2(D) \right ]_\theta=H_B^{m(1-\theta)}(D)\, ,$$
	and with the choice $\theta=1-\frac r m$ further that
	\begin{equation}\label{domLa}\text{dom}\lp \tilde\Lambda^{\frac r m} \rp=H_B^{r}(D)\, .\end{equation}
	Let $f\in L^2(D)$ be as in \eqref{l2decomposition}. Then, see e.g.\ \cite[(2.12)]{nochetto}, we have that
$\tilde\Lambda^{\frac r m}f=\sum_{j\s 1} \mu_j^{\frac r m} a_j(f) \Phi_j$,
	where $\mu_j=\lambda_j^{\frac m 2}$ is the $j^{th}$ eigenvalue of $\tilde\Lambda$. The previous sum converges in $L^2(D)$ if and only if 
	$\sum_{j\s 1}\lambda_j^r \left | a_j(f) \right |^2<+\infty$,
	so together with \eqref{domLa}, we obtain
	$$H_r(D)=\text{dom}\lp \tilde\Lambda^{\frac r m} \rp=H^r_B(D)\, .$$
	Therefore, by \cite[Théorème 8.1]{grisvard} and the discussion that follows, we have $H_r(D)\hookrightarrow H^r(D)$. 
	
	If $r\I \frac 1 2$, we have by \cite[(2.13)]{nochetto} 
	\begin{equation*}
		H_r(D)=\left \{\begin{array}{ll}H^r(D) & \text{if} \ r< \frac 1 2\, ,  \\ H_{00}^{\frac 1 2}(D) & \text{if} \ r=\frac 1 2\, , \end{array}\right.
	\end{equation*}
	where $H_{00}^{\frac 1 2}(D)$ is the Lions--Magenes space satisfying
	$H_{00}^{\frac 1 2}(D)\hookrightarrow H_{0}^{\frac 1 2}(D)$
	by \cite[Chapitre 1, Théorème 11.7]{lions}. In addition, for any $r\s 0$, $H_{0}^{r}(D)\hookrightarrow H^{r}(D)$ and thus $H_r(D)\hookrightarrow H^r(D)$. Next, by \cite[Chapitre 1, Théorèmes 9.1, 9.2 and (7.1)]{lions}, there exists a constant $C$ such that any function $u\in H^r(D)$ is the restriction of a function $\tilde u \in H_r(\R^d)$ to $D$ with $\| \tilde u \|_{H_r(\R^d)}\I C \| u \|_{H^r(D)}$. Therefore, $H_r(D)\hookrightarrow H^r(D) \hookrightarrow H_r(\R^d)|_D$ for any $r\s 0$, and by duality, we have $H_r(\R^d)|_D\hookrightarrow H^r(D)\hookrightarrow H_r(D)$ for $r\I 0$.
\end{proof}
 

\subsubsection{Existence of a càdlàg solution in $H_{r}(D)$ with $r<-\frac d 2$}

\begin{theo}\label{cadlag_bdd_domain}
 The mild solution $u$ to \eqref{SHE} constructed in Proposition \ref{existencebdddomain} has a càdlàg modification  in $H_r(D)$ for any $r<-\frac d 2$.
\end{theo}
In contrast to the case $D=[0,\pi]$, the eigenfunctions of $-\Delta$ on a general domain $D$ in $\R^d$ may not be uniformly bounded, see \eqref{boundedestimate}. Thus, the proof of Theorem~\ref{cadlagsolution-general} does not extend to higher dimensions. Instead, we use \cite[Theorem~1]{eidelman} to write 
\begin{equation}
	\label{decompositiongreen}
	G_D(t;x,y)= g(t,x-y)+ H(t; x,y)\, ,
\end{equation}
where $g$ is the heat kernel on $\R^d$ and $H$ is a function such that for any $\varepsilon >0$,  $(t,x,y) \mapsto H(t;x,y)$ is smooth on $[0,T]\times D \times B_\varepsilon^c(\partial D)$, where $B_\varepsilon(\partial D)$ is the $\eps$-neighborhood of $\partial D$.
Away from the boundary $\partial D$, $g$ can be dealt with as in Theorem~\ref{cadlagsolution}, and $H$ is smooth and therefore easily handled. In order to control the behavior close to the boundary, the change of measure technique (see the Appendix and also the proof of Theorem~\ref{cadlagsolution}) is again fruitful.
\begin{proof}[Proof of Theorem~\ref{cadlag_bdd_domain}]
We may assume that $u$ satisfies \eqref{newsolutiondomain}. 
For $\eps>0$, split $u(t,x)$ into the sum of three terms:
\begin{align*}
	u_\eps^1(t,x)&:=\int_0^t \int_D g(t-s,x-y) \sigma(u(s,y)) \mathds 1_{y \in B_\varepsilon^c(\partial D)} \,L(\dd s,  \dd y)\,, \\
	u_\eps^2(t,x)&:=\int_0^t \int_D H(t-s;x,y) \sigma(u(s,y)) \mathds 1_{y \in B_\varepsilon^c(\partial D)} \,L(\dd s,  \dd y)\,,\\
	u^3_\eps(t,x)&:= \int_0^t \int_D G_D(t-s;x,y) \sigma(u(s,y)) \mathds 1_{y \in B_\varepsilon(\partial D)} \,L(\dd s,  \dd y)\, .
\end{align*}


By \eqref{moment-D}, the same proof as in Theorem~\ref{cadlagsolution} for $u^{1,1}$ and $u^3$ shows that $t\mapsto u_\eps^1(t,\cdot)$ has a càdlàg version in $\hs$ for $r<-\frac d 2$, where the spatial variable takes values in the whole space $\R^d$ rather than just $D$. Thus, as a process with $x\in D$, it has a càdlàg version in $H_r(D)$ for $r<-\frac d 2$ by Lemma~\ref{equivhsspace}. Regarding $u^2_\eps$, since $(t,x,y) \mapsto H(t;x,y)$ is smooth on $[0,T]\times D \times B_\varepsilon^c(\partial D)$, we can mimic the part of the proof of Theorem~\ref{cadlagsolution} concerning $u^{1,2}$ in order to get that $(t,x) \mapsto u_\eps^2(t,x)$ is jointly continuous, and in fact, uniformly continuous since $D$ is bounded. In particular, $t\mapsto u_\eps^2(t,\cdot)$ is continuous in $H_r(D)$ for any $r\I 0$. 
For the last term $u^3_\eps$, we want to show that it converges to $0$ in $H_r(D)$, uniformly in $t\in[0,T]$. As a first step, we have
\begin{align*}
 \left \| u_\eps^3 (t,\cdot) \right \|_{H_r}^2=\sum_{k\s 1} (1+\lambda_k)^r \lp a_k^\eps(t)\rp ^2\, ,
\end{align*}
where $a_k^\eps(t) := \int_D \Phi_k(x)u_\eps^3(t,x)\, \dd x$. As in \eqref{calc-akn} and \eqref{Fu11}, one can then show that
\begin{align*}
 \left | a_k^\eps(t) \right | 
 &\I C\sup_{t\in [0,T]}\left | \int_0^t \int_D \Phi_k(y) \sigma(u(s,y)) \mathds 1_{y \in B_\varepsilon(\partial D)}\,L(\dd s,  \dd y) \right |\, .
\end{align*}
Next, let $\Phi(x):=\big(\sum_{k\s 1} (1+\lambda_k)^r \Phi_k^2(x)\big)^{\frac 1 2}$, which belongs to $L^2(D)$ by \eqref{summable} since $\|\Phi_k\|_{L^2(D)} = 1$ and $r<-\frac d 2$. Hence, assuming without loss of generality that $p$ in \hyperlink{hyp2}{\textbf{(H')}} satisfies $1\I p< 1+\frac 2 d$, we have by Lemma~\ref{Daniell-explicit}:
\begin{align*}
	\|\Phi\sigma(u)\|_{L,p} &\I \|\Phi\sigma(u)\|_{L^B,p} + \|\Phi\sigma(u)\|_{L^M,p} \I |b| \left\| \int_0^T\int_D |\Phi(x)\sigma(u(t,x))|\,\dd t\,\dd x\right\|_{L^p} \\
	&\quad\qquad+ C\left\| \bigg(\int_0^T\int_D\int_{|z|\I 1} |\Phi(x)\sigma(u(t,x))z|^2\,J(\dd t,\dd x,\dd z)\bigg)^{\frac 1 2}\right\|_{L^p} \\
	&\I \int_0^T\int_D \Phi(x) \|\sigma(u(t,x))\|_{L^p}\,\dd t\,\dd x+C \left(\int_0^T\int_D \Phi^p(x) \E[|\sigma(u(t,x))|^p]\,\dd t\,\dd x\right)^{\frac 1 p}\\
	&\I C (\| \Phi\|_{L^1(D)} + \|\Phi\|_{L^p(D)}) < +\infty\,.
\end{align*}
Thus, by Theorem~\ref{change}, there exists an equivalent probability measure $\Q$ such that
\begin{equation} \label{Qnorm} \|\Phi\sigma(u)\|_{L,2,\Q}<+\infty\,. \end{equation}
Furthermore, the Doob--Meyer decomposition of $L$ under $\Q$ is given by $L={L^{B,\Q}}+ L^{M,\Q}$, where
\begin{align*} {L^{B,\Q}}(\dd t,\dd x) &= b^\Q(t,x)\,\dd t\,\dd x=  \lp b + \int_{|z|\I 1} (Y(t,x,z)-1)\,\nu(\dd z) \rp \,\dd t\,\dd x\, ,\\
L^{M,\Q}(\dd t,\dd x) &= \int_{|z|\I 1} z \,\tilde J^\Q(\dd t,\dd x,\dd z)\, ,\quad \tilde J^\Q(\dd t,\dd x,\dd z) = J(\dd t,\dd x,\dd z) - Y(t,x,z)\,\dd t\,\dd x \,\nu(\dd z) \end{align*}
with some predictable random function $Y$, see \cite[Theorem~3.6]{chong_integrability}. By \cite[Theorem~4.14]{bichteler_jacod}, we deduce that $\|\Phi\sigma(u)\|_{L^{B,\Q},2,\Q}<+\infty$ and $\|\Phi\sigma(u)\|_{L^{M,\Q},2,\Q}<+\infty$. As a consequence,
\begin{equation}\label{fin} \E_\Q\left[\left(\int_0^T\int_D \Phi(x)|\sigma(u(t,x))b^\Q(t,x)|\,\dd t\,\dd x \right)^2\right]< +\infty  \end{equation}
and 
\begin{equation}\label{fin2} \E_\Q\left[\int_0^T\int_D \int_{|z|\I 1} (\Phi(x)\sigma(u(t,x))Y(t,x,z))^2\,\dd t\,\dd x\,\nu(\dd z)\right] <+\infty\, . \end{equation}
We obtain
\begin{align*}
	\E_\Q\left[\left \| u_\eps^3 (t,\cdot) \right \|_{H_r}^2\right]&=\sum_{k\s 1} (1+\lambda_k)^r \E_\Q[(a_k^\eps(t))^2]\I C \sum_{k\s 1} (1+\lambda_k)^r \|\Phi_k\sigma(u)\mathds 1_{B_\eps(\partial D)}\|^2_{L,2,\Q}\\
	&\I C\sum_{k\s 1} (1+\lambda_k)^r \lp \|\Phi_k\sigma(u)\mathds 1_{B_\eps(\partial D)}\|^2_{L^{B,\Q},2,\Q} +  \|\Phi_k\sigma(u)\mathds 1_{B_\eps(\partial D)}\|^2_{L^{M,\Q},2,\Q}\rp\, .
\end{align*}
For the first term in the parenthesis, we use the Cauchy--Schwarz inequality to obtain 
\begin{align*} &\sum_{k\s 1} (1+\lambda_k)^r  \|\Phi_k\sigma(u)\mathds 1_{B_\eps(\partial D)}\|^2_{L^{B,\Q},2,\Q} \\
	&\qquad= \sum_{k\s 1} (1+\lambda_k)^r \E_\Q\left[\left( \int_0^T\int_D |\Phi_k(x)\sigma(u(t,x))\mathds 1_{x\in B_\eps(\partial D)}b^\Q(t,x)|\,\dd t\,\dd x\right)^2\right]\\
	&\qquad = \E_\Q\bigg[ \sum_{k\s 1} (1+\lambda_k)^r \int_0^T\int_D \int_0^T\int_D |\Phi_k(x)\Phi_k(y)\sigma(u(t,x))\sigma(u(s,y))\mathds 1_{x,y\in B_\eps(\partial D)}\\
	&\qquad\quad\qquad \times b^\Q(t,x)b^\Q(s,y)|\,\dd t\,\dd x\,\dd s\,\dd y\bigg]\\
	&\qquad \I \E_\Q\left[\int_0^T\int_D\int_0^T\int_D |\Phi(x)\Phi(y)\sigma(u(t,x))\sigma(u(s,y))\mathds 1_{x,y\in B_\eps(\partial D)}b^\Q(t,x)b^\Q(s,y)|\,\dd t\,\dd x\,\dd s\,\dd y\right]\\
	&\qquad = \E_\Q\left[\left(\int_0^T\int_D|\Phi(x)\sigma(u(t,x))\mathds 1_{x\in B_\eps(\partial D)}b^\Q(t,x)|\,\dd t\,\dd x\right)^2\right] \to 0
\end{align*}
as $\eps\to0$
by \eqref{fin} and dominated convergence. Similarly, \eqref{fin2} implies that
\begin{align*}
	&\sum_{k\s 1} (1+\lambda_k)^r  \|\Phi_k\sigma(u)\mathds 1_{B_\eps(\partial D)}\|^2_{L^{M,\Q},2,\Q} \\
	&\qquad \I C \sum_{k\s 1} (1+\lambda_k)^r  \E_\Q\left[\int_0^T \int_D \int_{|z|\I 1}(\Phi_k(x)\sigma(u(t,x)) Y(t,x,z))^2 \mathds 1_{x\in B_\eps(\partial D)} \,\dd t\,\dd x\,\nu(\dd z)\right]\\
	&\qquad = C\E_\Q\left[\int_0^T \int_D \int_{|z|\I 1}(\Phi(x)\sigma(u(t,x)) Y(t,x,z))^2 \mathds 1_{x\in B_\eps(\partial D)} \,\dd t\,\dd x\,\nu(\dd z)\right] \to 0
\end{align*}
as $\eps\to0$. Altogether, there exists a subsequence of $u^3_\eps(t,\cdot)$ that converges almost surely to $0$ in $H_r(D)$ for $r<-\frac d 2$, uniformly in $t\in[0,T]$, which completes the proof. 
\end{proof}


\section{Partial regularity of the solution}\label{partial-reg}

In Section~\ref{cadlag}, we have established the existence of a version such that $t\mapsto u(t,\cdot)$ has càdlàg paths in (local) fractional Sobolev spaces. The goal of the current section is to investigate the partial regularity of the solution, that is, the behavior of the partial functions $t\mapsto u(t,x)$ for fixed $x\in D$ and $x\mapsto u(t,x)$ for fixed $t\in[0,T]$. In the case where the Lévy noise $L$ has locally finite intensity (so $L$ is a compound Poisson noise), it is clear that almost surely, no jump will fall onto a fixed $t$- or $x$-section of the solution. Because the Green's function $G_D(t;x,y)$ is smooth outside $\{0\}\times\{(x,x)\colon x\in D\}$, the partial functions are continuous, and even smooth, in this case. However, a general Lévy noise can have infinitely many jumps on any compact subset of $[0,T]\times D$, which may even fail to be summable. Still they never lie on a fixed section, but may come arbitrarily close to it, so its regularity is unclear a priori. As we shall show, the answer critically depends on the Blumenthal--Getoor index of the noise (that is, the smallest $p$ for which $\int_{[-1,1]} |z|^p\,\nu(\dd z)$ is finite), and both continuous and locally unbounded sample paths may arise.

Throughout this section, we consider the stochastic heat equation \eqref{SHE} on a bounded $C^\infty$-regular domain or $D=\R^d$, with some Lipschitz continuous $\sigma\colon \R\to\R$ and some bounded continuous $u_0\colon \bar D\to\R$ that is zero on $\partial D$. Furthermore, let $L$ be a pure-jump Lévy white noise as in \eqref{noise} and $u$ be the mild solution constructed under the hypotheses in Propositions~\ref{existence}, \ref{existencerealline} or \ref{existencebdddomain}, respectively. In particular, if $D=\R^d$, we are given $p,q>0$ such that \hyperlink{hyp1}{\textbf{(H)}} is satisfied; and if $D$ is a bounded domain, there exists $p>0$ such that \hyperlink{hyp2}{\textbf{(H')}} holds. 

\subsection{Regularity in space at a fixed time}

\begin{theo}\label{continuityspacewholespace}
	In the setting described above, assume that $p<\frac 2 d$. 
	Then, for any $t\in [0,T]$, the process $x\mapsto u(t,x)$ has a continuous modification.
\end{theo}
\begin{proof} The solution $u$ is the stationary limit of the mild solution $u_N$ to the truncated equation defined in \eqref{truncatedmildinterval}, \eqref{truncatedmild} or \eqref{uN-D} with noise $L_N$ given in \eqref{LN-interval}, \eqref{LN} or \eqref{truncatednoisebdd}, respectively.
	Therefore, we can suppose that $u=u_N$ for some $N\s 1$, and for simplicity, we only consider $N=1$.
	We prove the claim using different approaches depending on the value of $p$. 
	
\vspace{0.25\baselineskip}

\noindent
\underline{$1<p<2$:} Notice that $1<p<2$ can only occur in $d=1$ because of the hypothesis $p< \frac 2 d$. However, we keep the exponent $d$ since we will use similar ideas in the next case.
Let $A>0$ be such that $x\in(-A,A)^d$, and split $u$ into seven parts according to \eqref{def-u} and \eqref{u3-split}, with obvious changes when $D$ is bounded. In this case, we further assume that $A>0$ is large enough such that $D\subset (-A,A)^d$. Clearly, $V$ is jointly continuous, and as shown in the proof of Theorem~\ref{cadlagsolution}, 
the same is true for $u^{1,2}\mathds 1_{[-A,A]^d}$, $u^{2,2}\mathds 1_{[-A,A]^d}$ and $u^{3,2}$ in the case $D=\R^d$, while they are zero if $D$ is bounded. Furthermore, almost surely, $u^{2,1}$ consists of finitely many jumps, none of which occur at time $t$, so $x\mapsto u^{2,1}(t,x)$ is smooth because the Green's function $G_D(t;x,y)$ is so for $t>0$.
It remains to consider $u^{1,1}+u^{3,1}$, for which we apply the Kolmogorov continuity criterion. We have
	\begin{equation*}
		\begin{aligned}
			&\E[|(u^{1,1}+u^{3,1})(t,x)- (u^{1,1}+u^{3,1})(t,z)|^p]\\
			&\qquad =\E\lc \left | b\int_0^t \int_{[-2A,2A]^d} \lp G_D(t-s; x,y) - G_D(t-s; z,y)\rp  \sigma (u(s,y))\, \dd s\,  \dd y \right. \right. \\
			&\qquad\qquad +\left . \left.\int_0^t \int_{[-2A,2A]^d} \lp G_D(t-s;x,y) - G_D(t-s; z,y) \rp \sigma (u(s,y)) \,L^M(\dd s , \dd y) \right |^p \rc\\
			&\qquad \I C \lp \E\lc \left | b\int_0^t \int_{[-2A,2A]^d}  \lp G_D(t-s;x,y) - G_D(t-s; z,y)\rp  \sigma (u(s,y)) \,\dd s  \,\dd y \right |^p \rc \right.\\
			&\qquad \qquad+\left. \E\lc \left | \int_0^t \int_{[-2A,2A]^d}  \lp G_D(t-s;x,y) - G_D(t-s; z,y) \rp \sigma (u(s,y))\, L^M(\dd s , \dd y) \right |^p \rc \rp 
		\end{aligned}
	\end{equation*} 
for any $x,z\in D$.
	Then, using Hölder's inequality and \eqref{p-moment} or \eqref{moment1},
	\begin{equation*}
		\begin{aligned}
			& \E\lc \left | b\int_0^t \int_{[-2A,2A]^d}  \lp G_D(t-s;x,y) - G_D(t-s; z,y)\rp  \sigma (u(s,y)) \,\dd s\,  \dd y \right |^p \rc \\
			&\qquad \I \lp \int_0^t \int_{[-2A,2A]^d}  \left | G_D(t-s;x,y) - G_D(t-s; z,y)\right | \E \lc \left |  \sigma (u(s,y)) \right |^p \rc \, \dd s\,  \dd y\rp\\
			&\qquad \qquad \times \lp \int_0^t \int_{[-2A,2A]^d}  \left | G_D(t-s;x,y)- G_D(t-s; z,y)\right | \, \dd s \, \dd y \rp^{p-1}\\
			&\qquad \I C \lp \int_0^t \int_{D}  \left | G_D(t-s;x,y) - G_D(t-s; z,y)\right | \, \dd s \, \dd y \rp^{p} \,. 
		\end{aligned}
	\end{equation*}
For $D=\R$, the last term is bounded by $C |x-z|^p$, see \cite[Lemme~A2]{loubert}; for $D=[0,\pi]$, we can take the power $p$ inside the integral by Hölder's inequality, and further assume that $\frac 3 2<p<2$ (by \eqref{p-moment} there is no harm in taking a larger value of $p$ on bounded domains). Then the upper bound becomes $C |x-z|^{3-p}$ by \cite[Lemma~B.1(a)]{bally}.
	For the martingale part, we have by \cite[Theorem~1]{marinelli} and \eqref{moment1}, 
	\begin{equation*}
		\begin{aligned}
			&\E\lc \left | \int_0^t \int_{[-2A,2A]^d} \lp G_D(t-s;x,y) - G_D(t-s; z,y)\rp \sigma (u(s,y)) \, L^M(\dd s , \dd y) \right |^p \rc\\
			&\qquad \I C \int_0^t \int_{[-2A,2A]^d} \int_{|z|\I 1} |z|^p \left | G_D(t-s;x,y) -G_D(t-s; z,y) \right |^p \E\lc\left| \sigma (u(s,y))\right |^p \rc \,\dd s\, \dd y \, \nu(\dd z)\\
			&\qquad \I C  \int_0^t \int_{D}  \left | G_D(t-s;x,y)- G_D(t-s; z,y)\right |^p \,  \dd s\, \dd y\,.
				\end{aligned}
	\end{equation*}
If $D=\R^d$, then according to \cite[Lemme~A2]{loubert}, this is bounded by
\[
			 \begin{cases} C |x-z|^p \, , &  \text{if} \ p<\frac 3 2\, , \\ C|x-z|^{\frac 3 2}\log\lp |x-z|\rp & \text{if} \ p=\frac 3 2\, , \\C |x-z|^{3-p} & \text{if} \ p>\frac 3 2\, ;\end{cases} \]
and if $D=[0,\pi]$, and we take $\frac 3 2<p<2$, then the upper bound we obtain is again $C|x-z|^{3-p}$ by \cite[Lemma~B.1(a)]{bally}.
	Thus, the Kolmogorov continuity criterion (see e.g.\ \cite[Chapter~1, Corollary~1.2]{spdewalsh}) ensures the existence of a continuous modification of $u^{1,1}+u^{3,1}$ in the space variable $x$.
	

\vspace{0.25\baselineskip}

\noindent
\underline{$0<p\I 1$:} 
We use the same decomposition of $u$ as above, except that we replace $b$ by $b_0$ and $L^M(\dd t,\dd x)$ by $\int_{|z|\I 1} z \,J(\dd t,\dd x,\dd z)$. The proofs for $V$, $u^{2,1}$ and $u^{2,2}$ are not affected by this change, and up to the modification indicated at the end of the proof of Theorem~\ref{cadlagsolution}, the proof for $u^{1,2}$ is not affected, either. Since $\int_{|z|\I 1} |z|\,\nu(\dd z) < +\infty$, the small jumps of $L$ are summable and $u^{1,1}$ is actually a sum of possibly infinitely many terms, each of which is continuous in $x$ because no jump occurs exactly at time $t$. 
Furthermore, by \cite[Corollary~3.2.8]{davies}, we have for $x_0\in D$,
	\begin{equation}\label{g-calc}
		\begin{aligned}
			&\E  \lc  \left( \int_0^t \int_{[-2A,2A]^d} \int_{|z|\I 1} \sup_{x: |x-x_0|\I 1} G_D(t-s;x,y)  \left | z  \sigma (u(s,y))\right | \,J(\dd s , \dd y, \dd z) \right )^p \rc \\
			&\qquad\I  \int_0^t \int_{[-2A,2A]^d} \int_{|z|\I 1} |z|^p \sup_{x: |x-x_0|\I 1} G^p_D(t-s;x,y)   \E\lc \left | \sigma (u(s,y)) \right |^p \rc \, \dd s \, \dd y \, \nu(\dd z) \\
			&\qquad\I C  \int_0^t \int_{\R^d}  \sup_{x: |x-x_0|\I 1} g^p(t-s, x-y)  \, \dd s \, \dd y \\
			&\qquad= C\lp  \int_0^t \int_{|y-x_0|\I1}  g^p(s, 0) \,\dd s  \, \dd y +\int_0^t \int_{|y-x_0|>1}\lp 4\pi s\rp^{-\frac{p d}{2} }e^{-\frac{p (|y-x_0|-1)^2}{4s}} \,\dd s  \, \dd y \rp \,,
		\end{aligned}
	\end{equation} 
	which is finite because $p<\frac 2 d$. So the sum defining $u^{1,1}$ converges locally uniformly in $x$, which implies that $x\mapsto u^{1,1}(t,x)$ is continuous almost surely. Recalling that $b_0 = 0$ for $p<1$ and $D=\R^d$, $u^{3,1}$ and $u^{3,2}$ are non-zero in this case only if $p=1$. Then $u^{3,2}(t,x)$ is jointly continuous in $(t,x)$, as shown in the proof of Theorem~\ref{cadlagsolution}. If $D$ is bounded, $u^{3,2}$ is zero. For $u^{3,1}$ we use the approximation sequence $u^{3,1}_n$ defined after \eqref{u3-split}. For each $n$, the process $x\mapsto u^{3,1}_n(t,x)$ is continuous because
	\[ |u^{3,1}_n(t,x)-u^{3,1}_n(t,x')|\I C \int_0^T\int_D |G_D(t-s;x,y)-G_D(t-s;x',y)|\,\dd s\,\dd y\to 0 \]
	as $x'\to x$, see \cite[Lemme~A2]{loubert} for $D=\R^d$ and the proof of \cite[Proposition~5]{sanzsole03} for bounded $D$.
	Hence, it suffices to prove that for fixed $t$, $u^{3,1}_n(t,x)$ converges to $u^{3,1}(t,x)$, locally uniformly in $x$. By dominated convergence, this can be reduced to showing
	\[ \int_0^t \int_{[-2A,2A]^d} \sup_{x: |x-x_0|\I 1} G_D(t-s;x,y)|\sigma(u(s,y))|\,\dd s\,\dd y<+\infty\quad\text{a.s.} \]
But this follows from \eqref{g-calc} together with \eqref{p-moment}, \eqref{moment1} and \eqref{moment-D}, respectively, by taking expectation.
\end{proof}
%

\begin{rem}\label{tructruc}
	In particular, any (tempered) $\alpha$-stable noise with $\alpha\in (0,\frac 2 d)$ (and $b_0=0$ if $D=\R^d$ and $\alpha<1$) satisfies the hypothesis of Proposition~\ref{continuityspacewholespace}. The same holds for (variance-)gamma noises for all $d\s1$, inverse Gaussian noises for $d=1,2,3$ and normal inverse Gaussian noises for $d=1$ (cf.\ \cite{cont}).
\end{rem}

The next theorem shows that the value $\frac 2 d$ in the previous theorem is essentially optimal.
\begin{theo}\label{label}
	Let $\sigma=1$ and suppose there is $\delta>0$ such that the Lévy measure of $L$ satisfies
	\begin{equation} \label{nu-alpha}\nu(\dd z)=\frac{f(z)}{|z|^{\alpha+1}} \,\dd z \end{equation} 
	for $z\in [-\delta, \delta]$, where $\alpha\in[\frac 2 d,(1+\frac 2 d)\wedge 2)$ and $f\colon [-\delta,\delta]\to [0,+\infty)$ is measurable with $f(0)\neq 0$ and 
	\begin{equation}\label{almostalphabdd}
		\int_{-\delta}^\delta \frac{|f(z)-f(0)|}{|z|^{\alpha+1}}|z|^r\,\dd z<+\infty
	\end{equation}
for some $0<r<\frac 2 d$.
Then for fixed $t\in [0,T]$, the path $x\mapsto u(t,x)$ is unbounded on any non-empty open subset of $D$ with probability $1$.
\end{theo}
\begin{proof} Fix $t\in [0,T]$. Since $V$ is continuous in $(t,x)$, it suffices to show the unboundedness of 
	$$Y(x)=\int_0^t \int_{D} G_D(t-s;x,y)\,L(\dd s, \dd y),\quad x\in D\,.$$
	We start with the case where $f$ is constant, that is, $L$ is an $\alpha$-stable noise. Then $(Y(x)\colon x\in D)$ is an $\alpha$-stable process given in the form of \cite[(10.1.1)]{taqqu} with $E=[0,T]\times D$ and control measure $\dd s \, \dd y$. We shall check that the necessary condition \cite[(10.2.14)]{taqqu} for sample path boundedness in \cite[Theorem~10.2.3]{taqqu} is not satisfied, in particular that for $x_0 \in D$, and $\delta$ such that $B_{x_0}(\delta) \subset D$,
		\begin{equation}
		\label{necessarycondition}
		\int_{0}^{t}\int_{D}  \lp \sup_{x\in B_{x_0}(\delta)} G_D(t-s,x,y) \rp^\alpha\, \dd s \, \dd y =+\infty\, .
	\end{equation}
	Indeed, by \cite[Theorem~2 and Lemma~9]{gaussianlowerbound}, for any $x,y \in B_{x_0}(\delta)$, 
	\begin{equation}\label{lowerbound}
		G_D(t-s,x,y)\s C g(t-s,x-y)\, ,
	\end{equation}
which implies that
	\begin{align*}
		\int_{0}^{t}\int_{D} \sup_{x\in B_{x_0}(\delta)} G_D^\alpha(t-s;x,y)\, \dd s \, \dd y 
		&\s C \int_{0}^{t}\int_{B_{x_0}(\delta)} \frac{1}{\lp 4\pi (t-s) \rp^{\frac{\alpha d}{2}}} \,\dd s \, \dd y= +\infty \, ,
	\end{align*}
	and \eqref{necessarycondition} is proved.
	In the case of general $f$, we write
	\begin{equation*}
		\begin{aligned}
			\nu(\dd z)  &= \nu_1(\dd z) - \nu_2(\dd z)+\nu_3(\dd z) +\nu_4(\dd z)\\
			&=\left(\frac{\lp f(z)-f(0)\rp_+}{|z|^{\alpha+1}} -  \frac{\lp f(z)-f(0)\rp_-}{|z|^{\alpha+1}}+ \frac{f(0)}{|z|^{\alpha+1}}\right)\mathds 1_{z \in [-\delta , \delta]}\,\dd z+\mathds 1_{z \in [-\delta , \delta]^c}\,\nu(\dd z)\, ,
		\end{aligned}
	\end{equation*}
	and decompose $L$ accordingly into $L_1-L_2+L_3+L_4$ such that for $1\I i\I 4$,  $L_i$ has Lévy measure $\nu_i$ and is independent of the other three parts. If $u_i$ solves the additive heat equation with driving noise $L_i$, then by \eqref{almostalphabdd} and Theorem~\ref{continuityspacewholespace}, for any fixed $t\in [0,T]$, $x\mapsto u_1(t,x)$, $x\mapsto u_2(t,x)$ and $x\mapsto u_4(t,x)$ each has a continuous version. And since the first part of the proof shows that $x\mapsto u_3(t,x)$ is unbounded on any open subset of $D$, the same property holds for $x\mapsto u(t,x)$.  
\end{proof}

\begin{rem} Taking $f\equiv 1$, Theorem~\ref{label} shows that for the solution $u$ to the heat equation with an additive $\alpha$-stable noise where $\frac 2 d \I \alpha <1+\frac 2 d$ (since $\alpha<2$, this can only occur for $d\s 2$), for any $t\in[0,T]$, $x\mapsto u(t,x)$ is unbounded on any non-empty open subset of $D$. The same holds true if $\nu$ has the form \eqref{nu-alpha} with $\alpha$ in the same range and some $f$ that is $(\alpha-\frac 2 d +\eps)$-Hölder continuous at $0$. For $d\s2$, this includes tempered stable Lévy noise (with stability index in the indicated range) and normal inverse Gaussian noises.
\end{rem}


\subsection{Regularity in time at a fixed space point}
\begin{theo}\label{fixedspacewholespace} In the set-up described at the beginning of Section~\ref{partial-reg}, assume that $0<p<1$. Then for any $x\in D$, 
	the process $t\mapsto u(t,x)$ has a continuous modification.
\end{theo}
We need the following elementary lemma.
\begin{lem}\label{maxheatkernel} If $g$ is the heat kernel \eqref{heat-kernel}, then there exists $C>0$ such that for every $T>0$, 
	$$\sup_{t\in [0,T]} g(t,x) = \begin{cases}  C T^{-{\frac d 2}} e^{-\frac{|x|^2}{4T}}\, , &  \text{if } T<\frac{|x|^2}{2d}\, ,\\ C|x|^{-d}\, ,& \text{if } T\s \frac{|x|^2}{2d}\, .\end{cases}$$
\end{lem}

\begin{proof}[Proof of Theorem~\ref{fixedspacewholespace}]
	Again, by a stopping time argument, it suffices to show the regularity of $u_N$ for any $N\s 1$, as defined in \eqref{truncatedmildinterval}, \eqref{truncatedmild} or \eqref{uN-D}, respectively. We only consider $N=1$, and decompose $u=V+u^{1,1}+u^{1,2}+u^{2,1}+u^{2,2}+u^3$ as in the part ``$0<p\I 1$'' of the proof of Theorem~\ref{continuityspacewholespace} (with $A>0$ such that $x\in(-A,A)^d$ and $D\subset (-A,A)^d$ if $D$ is bounded). There we explained that $V$, $u^{1,2}\mathds 1_{[-A,A]^d}$ and $u^{2,2}\mathds 1_{[-A,A]^d}$ are jointly continuous. The term $u^{2,1}$ is again a finite sum of weighted heat kernels, so $t\mapsto u(t,x)$ is smooth because none of the jumps falls exactly on a given $x\in D$. Moreover,  for $u^{1,1}$, we can use \cite[Théorème 2.2.2]{loubert} in the case $D=\R^d$ because $\E[|\sigma(u(t,x))|^p] = \E[|\sigma(u_1(t,x))|^p]$ is uniformly bounded on $[0,T]\times [-2A,2A]^d$. The proof also applies to bounded $D$ because the heat kernel $G_D$ is majorized by a multiple of the heat kernel on $\R^d$ by \cite[Corollary~3.2.8]{davies}. Finally, since $p<1$ and $b_0=0$ if $D=\R^d$, $u^3$ only needs to be considered for bounded $D$. With $1<r<1+\frac 2 d$ and $0\I t' < t \I T$, we can use Hölder's inequality, the moment bound \eqref{p-moment} or \eqref{moment-D} and \cite[Proposition~5]{sanzsole03} (and the proof therein) to deduce for every $\gamma\in(0,1)$, 
\begin{align*}
	\E[|u(t,x)-u(t',x)|^r] &\I \left(\int_0^t \int_D |G_D(t-s;x,y)-G_D(t'-s;x,y)|\,\dd s\,\dd y\right)^{r-1}\\
	&\quad\qquad \times \int_0^t \int_D |G_D(t-s;x,y)-G_D(t'-s;x,y)| \E[|\sigma(u(s,y))|^r] \,\dd s\,\dd y\\
	&\I \left(\int_0^t \int_D |G_D(t-s;x,y)-G_D(t'-s;x,y)|\,\dd s\,\dd y\right)^r \I C|t-t'|^{\gamma r}
\end{align*}
Thus, by choosing $\gamma$ close to $1$, the claim follows from Kolmogorov's continuity theorem.
\end{proof}

\begin{theo}\label{alphafixedspace}
	If $\sigma\equiv1$ and $\nu$ satisfies \eqref{nu-alpha} for some $\alpha\in[1,1+\frac 2 d)$, $\delta>0$, and some measurable $f\colon [-\delta,\delta]\to [0,+\infty)$ with $f(0)\neq 0$ and \eqref{almostalphabdd} for some $0<r<1$, then
	for any $x\in D$, the process $t\mapsto u(t,x)$ is unbounded on any non-empty open interval in $[0,T]$ with probability $1$. 
\end{theo}
\begin{proof} The proof is analogous to the proof of Theorem~\ref{label}. It suffices to show that
	\begin{equation}
		\label{necessarycondition3}
		\int_{t_1}^{t_2}\int_{D}  \lp \sup_{t\in [t_1, t_2]}  G_D(t-s;x,y) \rp^\alpha\,\dd s \, \dd y =+\infty
	\end{equation}
for $x\in D$ and $0\I t_1 < t_2 \I T$.
By \cite[Theorem~2 and Lemma~9]{gaussianlowerbound}, it suffices to consider $D=\R^d$. In this case,
the integral above is bounded from below by 
	\begin{align*}
		\int_{t_1}^{t_2}\int_{\R^d} \sup_{t\in [t_1,t_2]}  g(t-s,x-y)^\alpha\, \dd s \, \dd y \s \int_{0}^{t_2-t_1}\int_{\R^d} \sup_{v\in [0,s]}  g(v,x-y)^\alpha \,\dd s \, \dd y \, ,
	\end{align*}
	so \eqref{necessarycondition3} follows from Lemma~\ref{maxheatkernel}:
	$$ \int_{t_1}^{t_2}\int_{\R^d} \lp \sup_{t\in [t_1, t_2]}  g(t-s,x-y) \rp^\alpha \dd s \, \dd y \s \int_0^{t_2-t_1} \int_{|x-y| \I \sqrt{2ds}}  \frac{C}{|x-y|^{d\alpha}} \dd s \, \dd y=+\infty\, .$$
\end{proof}

\begin{rem}\label{remalpha1} In contrast to the results on regularity in space, the critical exponent $p=1$ for temporal regularity does not depend on the dimension $d$. In particular, for all $d\s 1$, we obtain continuity of $t\mapsto u(t,x)$, $x$ fixed, for any (tempered) $\alpha$-stable noise with $\alpha\in (0,1)$ (and $b_0=0$ if $D=\R^d$), any (variance-)gamma noise and inverse Gaussian noise. In the case of additive noise, $t\mapsto u(t,x)$ is almost surely unbounded on any non-empty open subinterval for (tempered) $\alpha$-stable noises with $\alpha\in[1,1+\frac 2 d)$ and normal inverse Gaussian noises.
\end{rem}

\appendix

\section{Appendix: Integration with respect to random measures}\label{appB}

Just as Lévy processes are special instances of semimartingales, 
a Lévy noise as in \eqref{noise} is a random measure as introduced in \cite{bichteler_jacod}. We give a short introduction into the integration theory associated to it, hereby concentrating on the main ideas and results that we need for our purposes. All details not mentioned or explained can be found in \cite{bichteler_jacod, chong_integrability, lebedev2, lebedev}. Given a filtered probability space $(\Omega,\F,(\F_t)_{t\in [0,T]},\bbp)$ satisfying the usual conditions, consider a Polish space $E$ (e.g., $E=D$, the spatial domain in \eqref{SHE}), and denote by $\calp$ the $\sigma$-field $\calp_0\otimes \mathcal B(E)$, where $\calp_0$ is the usual predictable $\sigma$-field. With an abuse of terminology, $\calp$-measurable mappings from $\tilde \Omega :=\Omega\times [0,T]\times E$ to $\R$ are again called \emph{predictable} and their collection again denoted by $\calp$.

Given a sequence $(\tilde \Omega_k)_{k\s 1}$ in $\calp$ satisfying $\tilde\Omega_k \uparrow \tilde\Omega$, a mapping $M\colon \calp_M:=\bigcup_{k\s 1} \calp|_{\tilde\Omega_k} \to L^p$ where $p\in[0,+\infty)$ is called an \emph{$L^p$-random measure} if for every sequence $(A_i)_{i\s 1}$ of pairwise disjoint sets in $\calp_M$ with $\bigcup_{i\s1} A_i\in \calp_M$, we have $M(\bigcup_{i\s1} A_i) = \sum_{i\s1} M(A_i)$ in $L^p$, and some additional  ``$(\F_t)_{t\in[0,T]}$-adaptedness'' conditions are satisfied. In our example of a Lévy noise on $[0,T]\times D$, we can take $\tilde\Omega_k = D$ if $D$ is bounded, and $\tilde \Omega_k = [-k,k]^d$ if $D=\R^d$.

The stochastic integral of a \emph{simple integrand} of the form $S=\sum_{i=1}^r a_i \mathds 1_{A_i}$, where $r\in\N$, $a_i\in\R$ and $A_i\in\calp_M$, is defined in the canonical way by 
\[ \int_0^T\int_E S(t,x)\,M(\dd t,\dd x) := \sum_{i=1}^r a_i M(A_i)\,. \]
Denoting by $\cals_M$ the collection of such simple integrands, the extension of the integral to a larger subset of $\calp$ is carried out using the \emph{Daniell mean} 
\begin{equation}\label{Daniell} \|H\|_{M,p} :=\|H\|_{M,p,\bbp}:= \sup_{S\in\cals_M, |S|\I |H|} \left\| \int_0^T\int_E S(t,x)\,M(\dd t,\dd x) \right\|_{L^p}\, ,\quad H\in\calp\,. \end{equation}
 A predictable process $H$ is called \emph{$p$-integrable with respect to $M$} if there exists a sequence $(S_n)_{n\s 1}$ of simple integrands with $\| H-S_n \|_{M,p} \to 0$ as $n\to+\infty$. The collection of $p$-integrable processes is denoted by $L^{1,p}(M)$ (or $L^{1,p}(M,\bbp)$ if we want to emphasize the probability measure). The stochastic integral of $H$ with respect to $M$ is then defined as the $L^p$-limit of $\int_0^T\int_E S_n(t,x)\,M(\dd t,\dd x)$ (which exists and does not depend on the choice of $S_n$). In all notions introduced, the prefix $p$ is suppressed if $p=0$.
The constructed integral obeys the dominated convergence theorem, see \cite[(2.6)]{bichteler_jacod}.
\begin{theo}\label{domcon}
	If $(H_n)_{n\s 1}$ are predictable and converge pointwise to $H$, and $|H_n|\I H_0$ for all $n\s 1$ and some $H_0\in L^{1,p}(M)$, then $H_n, H \in L^{1,p}(M)$ and $\|H-H_n\|_{M,p}\to 0$ as $n\to+\infty$.
\end{theo} 

 In this paper, we are particularly interested in the case where $M$ is a linear combination of random measures of one of the following forms:
 \begin{enumerate}
 	\item[(a)] $M$ is a \emph{predictable strict random measure}, that is, almost every realization of $M$ is a measure on $[0,T]\times E$ and $t\mapsto \int_0^T \int_E \mathds 1_A(s,y) \mathds 1_{[0,t]}(s)\,M(\dd s,\dd y)$ is a predictable process for all $A\in\calp_M$.
 	\item[(b)] $M(\dd t,\dd x) = \int_{z\in \R} W(t,x,z)\,\tilde J(\dd t,\dd x,\dd z)$, where $J$ is an $(\F_t)_{t\in[0,T]}$-Poisson random measure with intensity measure $\nu(\dd t,\dd x,\dd z)$, $\tilde J = J-\nu$, and $W\mathds 1_{\tilde\Omega_k}$ is $1$-integrable with respect to $\tilde J$ (a random measure on $E\times\R$) in the sense above for every $k\s 1$.
 	\item[(c)] $M$ is a strict random measure of the form $M(\dd t,\dd x) = \int_{z\in\R} W(t,x,z)\,J(\dd t,\dd x,\dd z)$ where $W\mathds 1_{\tilde\Omega_k}$ is integrable with respect to $J$ for every $k\s 1$.
 \end{enumerate}
 
In these cases, the Daniell mean can be computed (or estimated) explicitly.
 \begin{lem}\label{Daniell-explicit} Let $\|X\|_{L^p} = \E[|X|^p]$ for $0<p<1$ and $\|X\|_{L^p} = (\E[|X|^p])^{\frac 1 p}$ be the usual $L^p$-norm for $p\s 1$.
 	\begin{enumerate}
 			\item In the case (a) above, we have  for every $0<p<+\infty$ and $H\in\calp$,
 		\begin{equation}\label{finitevarcase} \|H\|_{M,p} = \left\| \int_0^T \int_E |H(t,x)|\,|M|(\dd t,\dd x) \right\|_{L^p}\, , \end{equation}
 		where $|M|$ is the total variation measure of $M$.
 		\item  		 In the case (b) above, there exist for every $p\s 1$ constants $c_p, C_p>0$ such that for all $H\in\calp$,
 		\[ c_p \left\| \left( \int_0^T \int_E H^2(t,x)\,[M](\dd t,\dd x) \right)^{\frac 1 2} \right\|_{L^p} \I \|H\|_{M,p}\I C_p  \left\| \left( \int_0^T \int_E H^2(t,x)\,[M](\dd t,\dd x) \right)^{\frac 1 2} \right\|_{L^p}\,, \]
 		where $[M](\dd t,\dd x) = \int_{z\in \R} W^2(t,x,z)\,J(\dd t,\dd x,\dd z)$ is the quadratic variation measure of $M$. In particular,
 		\begin{equation}\label{p-est} \|H\|_{M,p}\I C_p \lp\int_0^T\int_E\int_\R \| H(t,x) W(t,x,z)\|_{L^p}^p\,\nu(\dd t,\dd x,\dd z)\rp^{\frac 1 p}\,. \end{equation}
 		\item In the case (c) above, we have for $0<p\I 1$ and $H\in\calp$,
 		\[ \|H\|_{M,p} \I \int_0^T \int_E \int_\R \| H(t,x)W(t,x,z)\|_{L^p}\, \nu(\dd t,\dd x,\dd z)\,. \]
 	\end{enumerate}
 \end{lem}
\begin{proof} For the first statement, the ``$\I$''-part follows from $$\left|\int_0^T\int_E S(t,x)\,M(\dd t,\dd x)\right| \I \int_0^T\int_E |S(t,x)|\,|M|(\dd t,\dd x) \I \int_0^T\int_E |H(t,x)|\,|M|(\dd t,\dd x)$$ for all $S$ with $|S|\I |H|$. For the ``$\s$''-part, observe that the right-hand side of \eqref{finitevarcase} equals $\|H\|_{|M|,p}$ by dominated convergence. Next, consider the measure $\mu(\dd \omega,\dd t,\dd x) = M(\omega,\dd t,\dd x)\,\bbp(\dd \omega)$ on $\calp$ and let $D(\omega,t,x)$ be its Radon--Nikodym derivative with respect to $|\mu|(\dd \omega,\dd t,\dd x)$. Then $D$ is predictable, $|D|\equiv 1$ and $|M|(\omega,\dd t,\dd x)=D(\omega,t,x)\,M(\omega,\dd t,\dd x)$. Hence,
	\begin{align*} & \sup_{S\in\cals_M, |S|\I |H|} \left\| \int_0^T\int_E S(t,x)\,|M|(\dd t,\dd x) \right\|_{L^p}= \sup_{S\in\cals_M, 0\I S\I |H|} \left\| \int_0^T\int_E S(t,x)\,|M|(\dd t,\dd x) \right\|_{L^p} \\ &\qquad= \sup_{S\in\cals_M, 0\I S\I |H|} \left\| \int_0^T\int_E S(t,x)D(t,x)\,M(\dd t,\dd x) \right\|_{L^p}\\ &\qquad \I \sup_{S\in\cals_M, 0\I S\I |H|} \left\| \int_0^T\int_E S(t,x)\,M(\dd t,\dd x) \right\|_{L^p} \I \|H\|_{M,p}\, , \end{align*}
	and \eqref{finitevarcase} is proved 
	
	For the second statement, we observe that $t\mapsto \int_0^T\int_E S(s,y)\mathds 1_{[0,t]}(s)\,M(\dd s,\dd y)$ is a local martingale for all $S\in\cals_M$, see \cite[Proposition~4.9(b)]{bichteler_jacod}. So the statement for $S\in\cals_M$ follows from the Burkholder--Davis--Gundy inequalities. The general case is again a consequence of the dominated convergence theorem. Inequality \eqref{p-est} can be proved along the lines of the third statement, which we now establish.
	
	Let $(T_i,X_i,Z_i)_{i\s 1}$ be the points of $J$ in $[0,T]\times E\times \R$. Then, using $(x+y)^p \I x^p+y^p$ for $x,y>0$ and $0<p\I 1$, we get
	\begin{align*} &\E\left[ \left| \int_0^T \int_E S(t,x)\,M(\dd t,\dd x) \right|^p \right] = \E\left[ \left| \sum_{i\s 1} S(T_i,X_i)W(T_i,X_i,Z_i) \right|^p \right]\\
		&\qquad \I \E\left[  \sum_{i\s 1} | S(T_i,X_i)W(T_i,X_i,Z_i) |^p \right] \I \E\left[  \int_0^T \int_E \int_\R |H(t,x)W(t,x,z)|^p\,J(\dd t,\dd x,\dd z) \right] \\
	&\qquad =\E\left[  \int_0^T \int_E \int_\R |H(t,x)W(t,x,z)|^p\,\nu(\dd t,\dd x,\dd z) \right]  \,, \end{align*}
and the proof is complete.
\end{proof}
 
With the help of the Daniell mean, one can obtain the following stochastic Fubini theorem.
 \begin{theo}\label{fubini} Let $(A,\mathcal A,\mu)$ be a $\sigma$-finite measure space, $M$ be an $L^p$-random measure for some $p>0$, and $H$ be a $\calp\otimes \mathcal A$-measurable function.
 	\begin{enumerate} \item If $p\s 1$ and $\int_A \| H(\cdot,\cdot,a)\|_{M,p}\,\mu(\dd a)<+\infty$, then 
 	\begin{equation*} \int_A \lp \int_0^T \int_E H(t,x,a)\, M(\dd t,\dd x)\rp \,\mu(\dd a) \quad \text{and}\quad \int_0^T \int_E \lp \int_A H(t,x,a)\,\mu(\dd a) \rp \,M(\dd t,\dd x) \end{equation*}
 	are equal almost surely, and all integrals involved are well defined.
 	\item If $0<p\I 1$ and $M$ is a random measure as in (c) above, the conclusion of the first part continues to hold if 
 	\begin{equation}\label{p-est-2} \int_0^T\int_E \int_\R \left\| \int_A |H(t,x,a)|\,\mu(\dd a) |W(t,x,z)| \right\|_{L^p} \,\nu(\dd t,\dd x,\dd z)<+\infty\,. \end{equation}
 	\end{enumerate}
 \end{theo}
\noindent The first part has been proved in \cite[Theorem~2]{lebedev} for $p=1$ (see also \cite{Bichteler95} for processes indexed only by time), but is obviously also valid for $p>1$ by the monotonicity of $L^p$-norms. In particular, Lemma~\ref{Daniell-explicit} can be used to verify the integrability assumption. The second part follows from the ordinary Fubini theorem ($M$ is a strict random measure here) together with an argument as in the proof of third part of Lemma~\ref{Daniell-explicit}.

A last result that we need relates to the possibility of recovering $L^2$-integrability from $L^p$-integrability, $0\I p<2$, upon an equivalent change of probability measure. For semimartingales, this result is well known, see \cite[Chapter~IV, Theorem~34]{protter}, for example. For random measures, it is proved in \cite[Corollary of Theorem~2]{lebedev2}.
\begin{theo}\label{change}
	If $M$ is an $L^p$-random measure and $H\in L^{1,p}(M)$ for some $0\I p<2$, then there exists a probability measure $\Q$ that is equivalent to $\bbp$ on $\mathcal F$ such that $\frac{\dd \Q}{\dd \bbp}$ is bounded, $\frac{\dd \bbp}{\dd \Q} \in L^{\frac{p}{2-p}}(\bbp)$, $M$ is an $L^2$-random measure under $\Q$, and $H\in L^{1,2}(M,\Q)$.
\end{theo}

\subsection*{Acknowledgements}
We thank two anonymous referees for their careful reading of our manuscript.
 
\bibliography{biblio}

\begin{thebibliography}{10}

\bibitem{Balan14}
Balan, R.~M.
\newblock {SPDE}s with {$\alpha$}-stable {L}{\'e}vy noise: {A} random field
  approach.
\newblock {\em Int. J. Stoch. Anal.}, 2014.
\newblock Article ID 793275, 22 pages.

\bibitem{Balan15}
Balan, R.~M.
\newblock Integration with respect to {L}{\'e}vy colored noise, with
  applications to {SPDE}s.
\newblock {\em Stochastics}, 83(3):363--381, 2015.

\bibitem{bally}
Bally, V., Millet, A., and Sanz-Sol{\'e}, M.
\newblock Approximation and support theorem in {H}\"older norm for parabolic
  stochastic partial differential equations.
\newblock {\em Ann. Probab.}, 23(1):178--222, 1995.

\bibitem{bichteler_jacod}
Bichteler, K. and Jacod, J.
\newblock Random measures and stochastic integration.
\newblock In Kallianpur, G., editor, {\em Theory and Application of Random
  Fields}, pp. 1--18. Springer, Berlin, 1983.

\bibitem{Bichteler95}
Bichteler, K. and Lin, S.~J.
\newblock On the stochastic {F}ubini theorem.
\newblock {\em Stoch. Stoch. Rep.}, 54(3--4):271--279, 1995.

\bibitem{Bogachev}
Bogachev, V.~I.
\newblock {\em Measure Theory. Volume I}.
\newblock Springer, Berlin, 2007.

\bibitem{Brzezniak10b}
Brze{\'z}niak, Z., Goldys, B., Imkeller, P., Peszat, S., Priola, E., and
  Zabczyk, J.
\newblock Time irregularity of generalized {O}rnstein--{U}hlenbeck processes.
\newblock {\em C. R. Acad. Sci. Paris, Ser. I}, 348(5--6):273--276, 2010.

\bibitem{Brzezniak09}
Brze{\'z}niak, Z. and Hausenblas, E.
\newblock Maximal regularity for stochastic convolutions driven by {L}{\'e}vy
  processes.
\newblock {\em Probab. Theory Relat. Fields}, 145(3--4):615--637, 2009.

\bibitem{Brzezniak10}
Brze{\'z}niak, Z. and Zabczyk, J.
\newblock Regularity of {O}rnstein--{U}hlenbeck processes driven by a
  {L}{\'e}vy white noise.
\newblock {\em Potential Anal.}, 32(2):153--188, 2010.

\bibitem{chen_dalang}
Chen, L. and Dalang, R.~C.
\newblock H\"older-continuity for the nonlinear stochastic heat equation with
  rough initial conditions.
\newblock {\em Stoch. Partial Differ. Equ. Anal. Comput.}, 2(3):316--352, 2014.

\bibitem{Chong2016}
Chong, C.
\newblock L{\'e}vy-driven {V}olterra equations in space and time.
\newblock {\em J. Theoret. Probab.}, 30(3):1014--1058, 2017.

\bibitem{Chongheavytailed}
Chong, C.
\newblock Stochastic {PDE}s with heavy-tailed noise.
\newblock {\em Stochastic Process. Appl.}, 127(7):2262--2280, 2017.

\bibitem{chong_integrability}
Chong, C. and Kl{\"u}ppelberg, C.
\newblock Integrability conditions for space--time stochastic integrals:
  {T}heory and applications.
\newblock {\em Bernoulli}, 21(4):2190--2216, 2015.

\bibitem{cont}
Cont, R. and Tankov, P.
\newblock {\em Financial Modelling with Jump Processes}.
\newblock Chapman \& Hall, Boca Raton, FL, 2004.

\bibitem{dalang99}
Dalang, R.~C.
\newblock Extending the martingale measure stochastic integral with
  applications to spatially homogeneous s.p.d.e.'s.
\newblock {\em Electron. J. Probab.}, 4(6):29 pages, 1999.

\bibitem{davies}
Davies, E.~B.
\newblock {\em Heat Kernels and Spectral Theory}.
\newblock Cambridge University Press, Cambridge, 1990.

\bibitem{eidelman}
\`E\u{\i}del'man, S.~D. and Ivasi\v{s}en, S.~D.
\newblock Investigation of the {G}reen matrix for a homogeneous parabolic
  boundary value problem.
\newblock {\em Trans. Moscow Math. Soc.}, 23:179--242, 1970.

\bibitem{skorohod}
Gihman, I.~I. and Skorohod, A.~V.
\newblock {\em The Theory of Stochastic Processes. {I}}.
\newblock Springer, New York, 1974.

\bibitem{Gorenflo14}
Gorenflo, R., Kilbas, A.~A., Mainardi, F., and Rogosin, S.~V.
\newblock {\em Mittag-Leffler Functions, Related Topics and Applications}.
\newblock Springer, Berlin, 2014.

\bibitem{grisvard}
Grisvard, P.
\newblock Caract\'erisation de quelques espaces d'interpolation.
\newblock {\em Arch. Ration. Mech. Anal.}, 25(1):40--63, 1967.

\bibitem{Hausenblas05}
Hausenblas, E.
\newblock Existence, uniqueness and regularity of parabolic {SPDE}s driven by
  {P}oisson random measure.
\newblock {\em Electron. J. Probab.}, 10(46):1496--1546, 2005.

\bibitem{nualart_huang}
Hu, Y., Huang, J., Nualart, D., and Tindel, S.
\newblock Stochastic heat equations with general multiplicative {G}aussian
  noises: {H}\"older continuity and intermittency.
\newblock {\em Electron. J. Probab.}, 20(55):50 pages, 2015.

\bibitem{davar}
Khoshnevisan, D.
\newblock {\em Analysis of Stochastic Partial Differential Equations}.
\newblock American Mathematical Society, Providence, RI, 2014.

\bibitem{Kotelenez82}
Kotelenez, P.
\newblock A submartingale type inequality with applications to stochastic
  evolution equations.
\newblock {\em Stochastics}, 8(2):139--151, 1982.

\bibitem{lebedev2}
Lebedev, V.~A.
\newblock Behavior of random measures under filtration change.
\newblock {\em Theory Probab. Appl.}, 40(4):645--652, 1996.

\bibitem{lebedev}
Lebedev, V.~A.
\newblock The {F}ubini theorem for stochastic integrals with respect to
  ${L}^0$-valued random measures depending on a parameter.
\newblock {\em Theory Probab. Appl.}, 40(2):285--293, 1996.

\bibitem{lions}
Lions, J.-L. and Magenes, E.
\newblock {\em Probl\`emes aux limites non homog\`enes et applications. {V}ol.
  1}.
\newblock Dunod, Paris, 1968.

\bibitem{marinelli}
Marinelli, C. and R{\"o}ckner, M.
\newblock On maximal inequalities for purely discontinuous martingales in
  infinite dimensions.
\newblock In Donati-Martin, C., Lejay, A., and Rouault, A., editors, {\em
  S\'eminaire de {P}robabilit\'es {XLVI}}, pp. 293--315. Springer, Cham, 2014.

\bibitem{Mytnik03}
Mytnik, L. and Perkins, E.
\newblock Regularity and irregularity of $(1+\beta)$-stable super-{B}rownian
  motion.
\newblock {\em Ann. Probab.}, 31(3):1413--1440, 2003.

\bibitem{nochetto}
Nochetto, R.~H., Ot\'arola, E., and Salgado, A.~J.
\newblock A {PDE} approach to fractional diffusion in general domains: {A}
  priori error analysis.
\newblock {\em Found. Comput. Math.}, 15(3):733--791, 2015.

\bibitem{peszat}
Peszat, S. and Zabczyk, J.
\newblock {\em Stochastic Partial Differential Equations with L{\'e}vy Noise:
  An Evolution Equation Approach}.
\newblock Cambridge University Press, Cambridge, 2007.

\bibitem{Peszat13}
Peszat, S. and Zabczyk, J.
\newblock Time regularity of solutions to linear equations with {L}{\'e}vy
  noise in infinite dimensions.
\newblock {\em Stochastic Process. Appl.}, 123(3):719--751, 2013.

\bibitem{protter}
Protter, P.~E.
\newblock {\em Stochastic Integration and Differential Equations}.
\newblock Springer, Berlin, 2005.
\newblock Second edition. Version 2.1, corrected third printing.

\bibitem{loubert}
Saint Loubert~Bi{\'e}, E.
\newblock \'{E}tude d'une {EDPS} conduite par un bruit poissonnien.
\newblock {\em Probab. Theory Related Fields}, 111(2):287--321, 1998.

\bibitem{taqqu}
Samorodnitsky, G. and Taqqu, M.~S.
\newblock {\em Stable Non-{G}aussian Random Processes: Stochastic Models with
  Infinite Variance}.
\newblock Chapman \& Hall, Boca Raton, 1994.

\bibitem{sanzsole02}
Sanz-Sol{\'e}, M. and Sarr{\`a}, M.
\newblock H{\"o}lder continuity for the stochastic heat equation with spatially
  correlated noise.
\newblock In Dalang, R.~C., Dozzi, M., and Russo, F., editors, {\em Seminar on
  Stochastic Analysis, Random Fields and Applications III}. Birkh{\"a}user,
  Basel, 2002.

\bibitem{sanzsole03}
Sanz-Sol{\'e}, M. and Vuillermot, P.-A.
\newblock Equivalence and {H}{\"o}lder--{S}obolev regularity of solutions for a
  class of non-autonomous stochastic partial differential equations.
\newblock {\em Ann. Inst. H. Poincar{\'e} Probab. Statist.}, 39(4):703--742,
  2003.

\bibitem{gaussianlowerbound}
Van~den Berg, M.
\newblock A {G}aussian lower bound for the {D}irichlet heat kernel.
\newblock {\em Bull. London Math. Soc.}, 24(5):475--477, 1992.

\bibitem{spdewalsh}
Walsh, J.~B.
\newblock An introduction to stochastic partial differential equations.
\newblock In Hennequin, P.~L., editor, {\em \'{E}cole d'\'Et\'e de
  Probabilit\'es de {S}aint-{F}lour, {XIV}---1984}, pp. 265--439. Springer,
  Berlin, 1986.

\end{thebibliography}
\bibliographystyle{myplain}
 
\end{document}